\documentclass[11pt]{article}

\usepackage{amsmath}
\usepackage{amsfonts}
\usepackage{amssymb}
\usepackage{amsthm}
\usepackage{pdfpages}
\usepackage{listings}
\usepackage{pgfplots}
\usepackage{a4wide}
\usepackage{enumerate}
\usepackage{dsfont}

\newcommand{\eqn}{\begin{eqnarray}}
\newcommand{\eqnd}{\end{eqnarray}}

\newcommand{\dx}{\,\text{d$x$}}

\newcommand{\dS}{\,\text{d$S$}}

\newcommand{\dt}{\,\text{d$t$}}
\newcommand{\dy}{\,\text{d$y$}}

\newcommand{\dW}{\,\text{d$W$}}

\newcommand{\e}{{\rm{e}}}
\newcommand{\argmin}[1]{\rm{argmin}}

\newcommand{\diag}{{\rm diag}}

\newtheorem{thm}{Theorem}
\newtheorem{cor}[thm]{Corollary}

\newtheorem{lem}[thm]{Lemma}
\newtheorem{pro}[thm]{Proposition}
\newtheorem{defi}[thm]{Definition}
\newtheorem{rem}[thm]{Remark}
\newtheorem{ex}[thm]{Example}

\begin{document}

%
%

	
	\title{Error analysis of truncated expansion solutions to high-dimensional parabolic PDEs}
	\author{Christoph Reisinger\footnote{Mathematical Institute and Oxford-Man Institute for Quantitative Finance, University of Oxford, Andrew Wiles Building, Woodstock Road, Oxford, OX2 6GG, United Kingdom, \{reisinge,wissmann\}@maths.ox.ac.uk} \, and Rasmus Wissmann\footnotemark[\value{footnote}] \footnote{Research supported by EPSRC and Nomura via a CASE Award, the German National Academic Foundation, and St. Catherine's College, Oxford.}}
	\date{\today}
	\maketitle

	\begin{abstract}
	We study an expansion method for high-dimensional parabolic PDEs which constructs accurate approximate solutions
	by decomposition into solutions to lower-dimensional PDEs,
	and which is particularly effective if there are a low number of dominant principal components.
	The focus	of the present article is the derivation of sharp error bounds for the constant coefficient case and a first and second order approximation. We give a precise characterisation when these bounds hold for (non-smooth) option pricing applications and provide numerical results demonstrating that the practically observed convergence speed is in agreement with the theoretical predictions.
	\end{abstract}
	
\noindent
	{\bf Key words}: high-dimensional PDEs; asymptotic expansions; anchored ANOVA; financial derivative pricing
	
	

	


\section{Introduction}

High-dimensional partial differential equations play an important role in the modelling of many real-world phenomena. This is chiefly the case in finance and economics, where one is often concerned with probability densities and expectations of economical or financial time series modelled by multivariate stochastic processes. 
These financial applications are the main motivation for the analysis in this paper, but we anticipate that similar methods can be useful, e.g., in the context of Fokker-Planck equations for high-dimensional chemical systems.

The computational effort necessary to solve $N$-dimensional PDEs numerically with standard grid-based methods grows exponentially with $N$, a phenomenon referred to as the ``curse of dimensionality''. Even more sophisticated PDE methods tailored to high-dimensional approximation, such as those based on sparse grids (see \cite{BG04} for a survey), are typically not able to deal with practical problems where $N$ exceeds about five to eight (see \cite{heinecke12highly-parallel} for results with sparse finite elements and \cite{LO08,RW07} for the sparse grid combination technique).
A sparse wavelet method is proposed in \cite{von2004numerical}, which gives almost dimension-independent convergence rates for parabolic equations with non-smooth initial data, as they are typical for derivative pricing applications, where the Cauchy data are singular with a singularity located in a low co-dimensional manifold. Numerical experiments for Black-Scholes and stochastic volatility models are given in \cite{HKSW10}.

A radial basis function method for multi-dimensional Black-Scholes PDEs is presented in \cite{pettersson2008improved}, but only numerical results for one and two dimensions are presented.

It seems clear that if the approximation is based on tensor product meshes, a compression of the data has to take place in the solution process.
An active area of research is concerned with low rank tensor approximations of such functions (see, e.g., the survey paper \cite{grasedyck2013literature}), and in particular the representation in the so-called tensor train format. The method can be targeted specifically towards high-dimensional diffusion problems, motivated by applications in finance, and \cite{kazeev2013low} derive theoretical bounds on the rank of such approximations, which break the curse of dimensionality. We are, however, not aware of any successful practical application.


%

There is a sizeable and growing body of literature which deals with the feasibility of integration (cubature) in high dimensions. Several of the ideas from cubature are transferable to the solution of PDEs, especially in the situation where the PDE solution can be expressed as a high-dimensional integral via a known Green's function (typically the transition density of an underlying stochastic process).

We will use a PDE-based analysis in this paper. This is to allow for future extensions to situations where a Green's function is not known, or even non-linear and free boundary problems, which are not covered by the present analysis.
Although the setting in this paper -- i.e., of linear second order parabolic PDEs with constant coefficients --  is in the intersection of problems where PDE and cubature methods are applicable, we provide novel error bounds even in this case.

Our approach here is most closely related to cubature based on tensor product decompositions. We focus therefore on that segment of the literature.\footnote{Alternative
approaches are Monte Carlo methods or cubature on Wiener space \cite{LV04, GL11, LL12}, based on high-order integration rules derived by exact integration of multivariate polynomials of Brownian integrals.}
In particular, \cite{sloanetal2002a, sloanetal2002b} construct lattice rules by ``dimension-by-dimension'' integration.
This is especially effective if there is an ordering of dimensions by importance. For instance, \cite{wangsloan2005}
investigates the effective dimension of integration problems in finance using ANOVA decompositions, including numerical examples for path-dependent and multi-asset derivatives; \cite{griebelholtz10}
combines truncated ANOVA decompositions with dimension adaptive sparse grids \cite{GG03}, provides an error analysis in mixed Sobolev norms and presents numerical examples for Collateralized Mortgage Obligations and Asian options.
A general framework of decompositions into lower dimensional projections is developed in \cite{kuoetal2010}, who show certain minimality and uniqueness properties of such decompositions, which contain ANOVA as a special case.


A similar decomposition is the basis of a PCA-based expansion method for PDEs introduced in \cite{RW07}.
The method exploits the empirical observation that many diffusion processes of interest (such as stock prices in equity baskets, or forward interest rates with different maturities) have relatively large correlations, which leads to covariance matrices with one or a few significant principal components, while the eigenvalues for the remaining eigenvectors are an order of magnitude smaller. Transformed into the corresponding basis, the solution of the PDE can be reasonably well approximated by a low-dimensional PDE solution restricted to the first principal components. Then, dimensionwise corrections (interpretable as approximate Taylor expansions in the small eigenvalues) can be added to arrive at successively more accurate approximations. The key point is that even if the dimension of the original PDE may be very high (say 10 to 100), an approximation with practically sufficient accuracy can often be constructed by the solution of a sequence of much lower-dimensional PDEs (say 2- or 3-dimensional).

The relation between PDE expansions and anchored ANOVA is pointed out in
 \cite{R12} and \cite{SGW12};
a higher order extension is also sketched in \cite{HKSW10}.
A practically useful feature of decomposition methods is their inherent parallelism, which is exploited in \cite{RW07} and more recently in
\cite{SMW13}.
An extension to stochastic volatility models is given in \cite{HKSW10}, where option values for baskets with two to eight equities, each with their stochastic volatility process, are computed, albeit with significantly reduced accuracy in the higher-dimensional cases.
In \cite{RW13}, we demonstrated that these methods can be successfully applied to more complex financial market models and derivatives. For example, we were able to solve 60-dimensional PDEs for Bermudan Swaptions in the LIBOR Market Model 
with comparable accuracy to standard Monte Carlo methods in similar or less (often significantly less) computation time.

Another line of research concerns the optimal coordinate system as basis for
these expansions.
PCA-based expansions (PDE-motivated or anchored ANOVA) are well-suited to deal with many derivative pricing applications, because market assets typically show high levels of correlation, and therefore a transformation to principal components gives a natural ordering with rapidly decreasing importance of higher-dimensional contributions.
In \cite{IT04}, optimal linear transformations of the original coordinates are constructed
in order to minimise the effective dimension,
while \cite{griebelhullmann2014}
proposes non-linear
coordinate transformations to extend linear PCA.

In this article, we investigate the theoretical properties of the PCA-based PDE expansion method from \cite{RW07}. We provide a framework for establishing the existence and accuracy of the approximate solutions, and give precise theoretical error bounds for first and second order versions.
While the method has previously been motivated heuristically via Taylor-expansions and its relation to anchored ANOVA methods, and its usefulness has been demonstrated by successful numerical studies, these are to the authors' knowledge the first theoretical error bounds in terms of the PDE
coefficients and the data. 

Specifically, the contribution of this article is to
\begin{itemize}
\item
derive error bounds for the PCA-based expansion method for parabolic constant coefficient PDEs, where the initial data have certain mixed-order smoothness;
\item
analyse theoretically the applicability to non-smooth data which typically arise in financial engineering;
\item
demonstrate agreement between theoretical predictions and experimental results.
\end{itemize}

The rest of this paper is organised as follows. 
Section \ref{sec:ExpansionMethod} presents the PCA-based expansion method for the heat equation and states and discusses the first main theoretical result.
Section \ref{sec:ErrorBounds} proves sharp error bounds for a first- and second order expansion under sufficient regularity.
Section \ref{subsec:taylor} provides an alternative construction of the schemes via Taylor expansions and discusses the links.
Section \ref{subsec:existderiv} analyses the applicability to non-smooth data, focussing on cases arising in finance.
In Section \ref{sec:Numerics}, we first generalise the method to constant coefficient parabolic PDEs and then give numerical examples which demonstrate agreement between the theoretical predictions and numerical results.
Finally, Section \ref{sec:Conclusion} summarizes the results and describes future research directions.


    \section{Set-up and first main result}\label{sec:Expansion}
	\label{sec:ExpansionMethod}
\label{sec:setup}

To introduce the main concepts, we focus on the $N$-dimensional heat equation
\begin{eqnarray}\label{eq:PDEHeat}
        \frac{\partial u}{\partial t} &=& \mathcal{L} u = \sum_{k=1}^N \lambda_k\frac{\partial^2 u}{\partial z_k^2}, \\
		u(z,0) &=& g(z), \label{eq:PDEHeatBoundary}
		\end{eqnarray}
for $z\in\mathbb{R}^N$, $t\in (0,T)$, $\lambda = (\lambda_1,\ldots,\lambda_N) \in \mathbb{R}_+^N$, and assume $\lambda_1 \ge \lambda_2 \ge \ldots \ge 0$.

We will discuss in Section \ref{subsec:defmodel} how principal component analysis (PCA) can be used to transform a general parabolic PDE with constant coefficients to this form, and outline in
Section \ref{sec:Conclusion}
how localisation arguments can be used to deal with variable coefficients, with a reference to \cite{RW17} where numerical tests are presented for
the variable coefficient setting.
	
	
 	\begin{defi}
	\label{def:unu}
 	Given an initial-value problem of the form (\ref{eq:PDEHeat}) and (\ref{eq:PDEHeatBoundary}),
	and an index set $\nu \subseteq \{1,\ldots,N\}$, define
	a differential operator 
	 	\begin{eqnarray}
\mathcal{L}^\nu = \sum_{k\in\nu} \lambda_{k}\frac{\partial^2}{\partial z_k^2},		
\label{eqn:lnu}
 	\end{eqnarray}
%
%
%
and an approximation $u^\nu$ to $u$ as the solution of
     \begin{eqnarray}\label{eq:PDEDef4}
         \frac{\partial u^\nu}{\partial t}  &=& \mathcal{L}^\nu u^\nu, \\
 		u^\nu(\cdot,0) &=& g. \label{eq:PDEDef4boundary}
     \end{eqnarray}
 	\end{defi}
 	
    Note that since $\mathcal{L}^\nu$ only operates on the dimensions in the index set $\nu$, the problem of calculating $u^\nu(z_0,T)$ \emph{for a fixed} $z_0\in\mathbb{R}^N$ is of spatial dimension $|\nu|$. Calculating the full solution $u^\nu(z,T)$ for all values of $z\in\mathbb{R}^N$ is still an $N$-dimensional problem.
	
 	\begin{defi}
	\label{def:uxi}
 	For a given $\xi = \{ (w_1,\nu_1), (w_2,\nu_2), \ldots, (w_n,\nu_n) \}$, where $w_i\in\mathbb{R}$ and $\nu_i\subseteq \{1,\ldots,N\}$ for all $1\leq i\leq n$, we define
     \begin{eqnarray}\label{eq:PDEDef5}
         u^\xi &=& \sum_{(w,\nu)\in\xi}w \; u^\nu.
     \end{eqnarray}
     We will refer to $u^\xi$ as a \emph{(truncated) expansion}.
 	\end{defi}
	
 	The idea is that $\xi$ encodes an approximate solution for $u$ via expansion into solutions $u^\nu$ of lower-dimensional PDEs. 
		We will use the notation
\begin{eqnarray}
\label{eq:expan-err}
	\hat{u}^\xi &:=& u^\xi - u
\end{eqnarray}
for the expansion error. 



	In the following, we introduce the expansions considered in this paper.
	Consider the solution $u(z,t,\lambda)$ of (\ref{eq:PDEHeat}) explicitly as a function of $\lambda$ and define, for some $0 \le \lambda^0
	= (\lambda^0_1,\ldots,\lambda^0_N)\in \mathbb{R}^N$, $0 \le \delta \lambda = (\delta \lambda_1,\ldots,\delta\lambda_N) \in \mathbb{R}^N$, $e_j$ the $j$-th canonical basis vector,
	\begin{eqnarray*}
	(S_j u)(z,t,\lambda^0) &=& u(z,t,\lambda^0 + \delta\lambda_j e_j), \\
	(\Delta_j u)(z,t,\lambda^0) &=& u(z,t,\lambda^0 + \delta\lambda_j e_j)-u(z,t,\lambda^0),
	\end{eqnarray*}
such that $S_j = I + \Delta_j$. Then
	\begin{eqnarray}
	u(z,t,\lambda^0 + \delta \lambda) = \left(\prod_{j=1}^N S_j\right) u(z,t,\lambda^0) 
	= \left(\prod_{j=1}^N (I + \Delta_j)\right) u(z,t,\lambda^0) 
= \!\!\!\!\!\! \sum_{\alpha \in \{0,1\}^N} \Delta^\alpha u(z,t,\lambda^0),
	\label{anova1}
	\end{eqnarray}
where $\alpha$ is a multi-index of $0$'s and $1$'s, i.e., the difference operator in each direction appears at most with power 1 in each term, and
$\Delta^\alpha = \prod_{i=1}^N \Delta_i^{\alpha_i}$.
	For instance, for $N=2$, and omitting $z$ and $t$ for brevity,
		\begin{eqnarray*}
	u(\lambda^0_1+\delta \lambda_1,\lambda^0_2+\delta \lambda_2) &=& \hspace{-0.6 cm}
	\underbrace{u(\lambda^0_1,\lambda^0_2)}_{=\ \Delta_1^{0} \Delta_2^{0} u
	\ =\  \Delta^{(0,0)}u} \\
	&& \hspace{-3 cm} + \underbrace{u(\lambda^0_1+\delta \lambda_1,\lambda^0_2)- u(\lambda^0_1,\lambda^0_2)}_{=\ \Delta_1^{1} \Delta_2^{0} u
	\ =\  \Delta^{(1,0)}u}
		+ \underbrace{u(\lambda^0_1,\lambda^0_2+\delta \lambda_2) - u(\lambda^0_1,\lambda^0_2)}_{=\  \Delta_1^{0} \Delta_2^{1} u
	\ =\  \Delta^{(0,1)}u} \\ 
	&&  \hspace{-3 cm} + \underbrace{u(\lambda^0_1+\delta \lambda_1,\lambda^0_2+\delta \lambda_2)
	- u(\lambda^0_1+\delta \lambda_1,\lambda^0_2)
	- u(\lambda^0_1,\lambda^0_2+\delta \lambda_2)
	+u(\lambda^0_1,\lambda^0_2)}_{=\  \Delta_1^{1} \Delta_2^{1} u
	\ =\  \Delta^{(1,1)}u}.
	\end{eqnarray*}

	Now consider specifically $\lambda^0 = (\lambda_1,\ldots,\lambda_r,0,0,\ldots,0)$, $\delta\lambda = \lambda-\lambda^0$,
	then, for any $m\ge 0$,
		\begin{eqnarray}
		\label{eqn:truncexp}
	u(\lambda) \;\; = & \underbrace{\sum_{j=0}^m \sum_{|\alpha|=j}  \Delta^\alpha u(\lambda^0)}_{=:} &+ \;\;
	 \underbrace{\sum_{j=m+1}^N \sum_{|\alpha|=j} \Delta^\alpha u(\lambda^0)}_{=:} \\
	\;\; = & 
	 u_{r,m}(\lambda) &+ \hspace{1.7 cm} \hat{u}_{r,m}(\lambda),
	 \label{eqn:truncerr}
		\end{eqnarray}
	where $u_{r,m}$ is seen as an approximation to $u$, and $\hat{u}_{r,m}$ the error, $\alpha \in \{0,1\}^N$, $|\alpha| = \sum_{i=1}^N \alpha_i$.
Then for $\alpha \in \{0,1\}^N$
one finds that
	\begin{eqnarray}
		\Delta^\alpha u(\lambda^0) &=& 
		\sum_{0\le \beta\le \alpha} (-1)^{\alpha-\beta} u (\lambda^0 + \delta\lambda \cdot \beta)
		\qquad \text{ if }  \forall \; k \le r: \; \alpha_k=0,
\label{mixed-difference}
	\end{eqnarray}
	where `$\cdot$' is element-wise multiplication, and $\Delta^\alpha u(\lambda^0) = 0$ otherwise. The last statement follows because
	$\delta \lambda_k = 0$ for all $k\le r$. 
	To demonstrate (\ref{mixed-difference}), we proceed by induction in $|\alpha|$.
	The statement is clearly true for $\alpha=0$.
	Consider next $\alpha = e_j$ for some $j>r$, i.e.\ $|\alpha| = 1$. Then
\begin{eqnarray*}
	\Delta^\alpha u \; = \; u(\lambda_0 + \lambda_j e_j) - u(\lambda_0) \; = \; 
	u(\lambda_0 + \delta \lambda \cdot \alpha) - u(\lambda_0) \; = \; \sum_{0\le \beta\le \alpha} (-1)^{\alpha-\beta} u (\lambda^0 + \delta\lambda \cdot \beta).
\end{eqnarray*}	
Let now $0 \le \alpha$, $0\neq \alpha$ with $\alpha_j=0$ for some $j$.
Then, from the induction hypothesis,
\begin{eqnarray*}
	\Delta^{\alpha+e_j} u(\lambda^0) & = &
	\Delta_j \sum_{0\le \beta\le \alpha} (-1)^{\alpha-\beta} u (\lambda^0 + \delta\lambda \cdot \beta) \\
	 &=& 
	\sum_{0\le \beta\le \alpha} (-1)^{\alpha-\beta} u (\lambda^0 + (\delta\lambda + \lambda_j e_j) \cdot \beta)
	- \sum_{0\le \beta\le \alpha} (-1)^{\alpha-\beta} u (\lambda^0 + \delta\lambda \cdot \beta) \\
	&=&		\sum_{e_j \le \beta\le \alpha+e_j} (-1)^{\alpha-\beta+1} u (\lambda^0 + \delta\lambda \cdot \beta)
	+ \sum_{0\le \beta\le \alpha} (-1)^{\alpha-\beta+1} u (\lambda^0 + \delta\lambda \cdot \beta) \\ 
	&=& \sum_{0\le \beta\le \alpha+e_j} (-1)^{\alpha+e_j-\beta} u (\lambda^0 + \delta\lambda \cdot \beta).
\end{eqnarray*}

	We can make a number of observations:
	\begin{itemize}
	\item
	For $\lambda>0$,
	$\Delta^{\alpha}/\lambda^{\alpha}$ is a consistent approximation to the mixed derivative in $\lambda$ of order ${\alpha}$. 
	\item
$u_{r,m}$ is of the form $u^\xi$ from Definition \ref{def:uxi} for some $\xi$, i.e., a linear combination of solutions to (\ref{eq:PDEDef4}--\ref{eq:PDEDef4boundary}). We will use both notations interchangeably. 
\item
The largest number of different non-zero $\lambda_j$ for any such term is $r+m$, which implies that no PDE of dimension higher than $r+m$ has to be solved to find $u_{r,m}$.
\item
Lastly, (\ref{eqn:truncexp}) represents an anchored ANOVA decomposition of the function $u$, for a particular choice of anchor. See, for instance, \cite{griebelholtz10} or \cite{kuoetal2010}.
	\end{itemize}

\begin{ex}
\label{ex-first-order}
 For $m=1$, we can write $u_{r,1}$, using Definition \ref{def:unu}, as 
  \begin{eqnarray}
  \nonumber
 u_{r,1} &=& u^{\{1,\ldots,r\}} + \sum_{k=r+1}^N \left(u^{\{1,\ldots,r,k\}} - u^{\{1,\ldots,r\}} \right) \\
 &=& (1+r-N) u^{\{1,\ldots,r\}} + \sum_{k=r+1}^N u^{\{1,\ldots,r,k\}},
 \label{ur1}
  \end{eqnarray}
 such that, with Definition \ref{def:uxi}, $u_{r,1} = u^\xi$ for
  \begin{eqnarray}
\nonumber
 	\xi &=& \{ (1+r-N,\{ 1,\ldots,r\}), (1,\{ 1,\ldots,r,r+1\}), (1,\{ 1,,\ldots,r,r+2\}),
	\\ && (1,\{ 1,,\ldots,r,r+3\}),\ldots, (1,\{ 1,,\ldots,r,N\}) \}.
	 \label{xi-first}
 \end{eqnarray}
 \end{ex}

We now state the main result on the expansion error for Example \ref{ex-first-order}.
We will mainly be concerned with smooth solutions from the following class of functions.
\begin{defi}
\label{def:Cb}
Let
\begin{eqnarray*}
C^{j,k,mix} &=& \left\{
g \in C^b: \;
\partial_{i_1}^j \ldots \partial_{i_k}^j g \in C^b, \; \forall 1\le i_1 < \ldots < i_k \le N
\right\}, \\
C^b &=& \left\{
g: \,
\mathbb{R}^N \rightarrow \mathbb{R} \text{ continuous}: \;
\exists \ c>0 \;\;  \forall \ z \in \mathbb{R}^N \;\;
|g(z)| \le c  
\right\}.
\end{eqnarray*}
\end{defi}
	Then we have the following:
\begin{thm}\label{thm:BoundPDEDiff2}
Assume $g \in C^{2,2,mix}$ in (\ref{eq:PDEHeat}--\ref{eq:PDEHeatBoundary}).
Then the expansion error $\hat{u}_{r,1}$ 
satisfies
\begin{eqnarray}\label{eq:BoundPDEDiff2_1}
		\left\| \hat{u}_{r,1}(\cdot, t) \right\|_\infty &\leq& 
		t^2 \sum_{r<i<k\le N} { \lambda_k \lambda_i \left\| \frac{\partial^4 g}{\partial z_k^2\partial z_i^2} \right\|_\infty }.
\end{eqnarray}
\end{thm}
\begin{proof}
    See end of Section \ref{sec:FirstOrderFirstREigenvalues}.
\end{proof}
While published error bounds for ANOVA-type expansions tend to focus on $L_2$-type estimates of ANOVA terms
(see, e.g., \cite{wang08} or \cite{zhangetal12}), we use PDE theory to derive $L_\infty$ error bounds in terms of the mixed smoothness of the components not captured in the expansion.

\begin{rem}
    \label{rem:superposition1}
    There are no ``diagonal'' quadratic terms with factors $\lambda_k^2$ and univariate derivatives in (\ref{eq:BoundPDEDiff2_1}). 
    This has the important consequence that any solutions which only depend on one of the $z_k$ are integrated exactly (as the cross-derivatives vanish), and by superposition any linear combination of such terms. By similar reasoning (and as the second derivative of affine functions is zero), the error $\widehat{u}_{r,1}$ is zero for initial conditions of the form
    \begin{eqnarray*}
          g(z) = \sum_{k=r+1}^N \Bigg( g_k(z_1,\ldots,z_r,z_k) \;\; + \!\!
       \sum_{\scriptsize \begin{array}{c} i\!=\!r\!+\!1 \\i \!\neq\! k\end{array}}^N
      z_i \,  g_{i,k}(z_1,\ldots,z_r,z_k)\Bigg),
    \end{eqnarray*}
for any functions $g_k$ and $g_{i,k}$, $i,k=r+1,\ldots, N$, i.e., they are functions of $z_1,\ldots, z_r, z_k$ only; hence, the expression $(\ldots)$ is
affine in $z_j$, $j\notin \{1,\ldots,r,k\}$, but arbitrarily nonlinear as a function of $z_1,\ldots, z_r, z_k$.
If an initial condition can be approximated well by functions of this form, the expansion error $\widehat{u}_{r,1}$ will be small, irrespective of smoothness of the data or smallness of the eigenvalues and time. 
\end{rem}

The smoothness requirement $g \in C^{2,2,mix}$ can be considerably weakened using the following observations.
First, for the function $u^\nu$ from (\ref{eq:PDEDef4}, \ref{eq:PDEDef4boundary}), which only contains the diffusion terms relating to an index set $\nu \subseteq \{1,\ldots,N\}$,
	we have the Green's function representation
    \begin{eqnarray}\label{eq:PDEHeatSolutionNu}
        u^\nu(z,t) &=& \int_{\mathbb{R}^N}{\Phi^\nu(y,t) \, g(z-y) \; dy }, \\
        \Phi^\nu(y,t) &=& \prod_{k\in\nu}{ \frac{\exp( -{y_k^2}/{(4\lambda_k t)})}{\sqrt{4\pi \lambda_k t}} } \prod_{k\in\{1,\ldots,N\}\setminus\nu}{\delta(y_k)},\label{eq:HeatGreensNdimNu}
    \end{eqnarray}
	where $\delta(\cdot)$ is the Dirac delta function.
 We can therefore replace the initial condition $g$ in the estimate (\ref{eq:BoundPDEDiff2_1}) with a function $G$ which has been smoothed by the heat kernel in the first $r$ directions, namely, for $\nu = \{1,\ldots, r\}$,
\begin{eqnarray}\label{eq:DefG}
G(z,t) 
= u^{\nu}(z,t) =
\int_{\mathbb{R}^r} \prod_{k=1}^{r}{ \frac{\exp( -{y_k^2}/{(4\lambda_k t)})}{\sqrt{4\pi \lambda_k t}} } g(z_1-y_1,\ldots,z_r-y_r,z_{r+1},\ldots,z_n)  \; dy_1 \ldots dy_r.
\end{eqnarray}
The smoothness of $G$
is all that is relevant for the expansion error. Thus, even if $g$ is only, e.g., piecewise smooth, it is only in degenerate cases that $G$ is not smooth everywhere. We analyse this in Section \ref{subsec:existderiv} and, in particular, give illustrative examples in Appendix \ref{subsec:exnonsmooth}.
    
\begin{cor} \label{cor:BoundPDEDiff2}
Assume $G(\cdot,t) \in C^{2,2,mixed}$ (with $G$ from (\ref{eq:DefG})).
Then the expansion error $\hat{u}_{r,1}$ 
satisfies
\begin{eqnarray}
		\left\| \hat{u}_{r,1}(\cdot,t) \right\|_\infty &\leq& 
		t^2 \sum_{r<i<k\le N}  { \lambda_k \lambda_i \left\| \frac{\partial^4 G}{\partial z_k^2\partial z_i^2} (\cdot,t) \right\|_\infty }.
\end{eqnarray}
\end{cor}
\begin{proof}
We first note that for fixed $t$, the solution $v(x,t';t)$ of
\begin{eqnarray*}
\frac{\partial v}{\partial t'} &=& \sum_{k=r+1}^N \lambda_{k} \frac{\partial^2 v}{\partial z_k^2}, \\
v(\cdot,0) &=& G(\cdot,t)
 \end{eqnarray*}
agrees with $u$ for $t'=t$, $v(\cdot,t;t) = u(\cdot,t)$. Moreover, 
from (\ref{ur1}),
\begin{eqnarray*}
u_{r,1} = (1+r-N) v^{\{\}} + \sum_{k=r+1}^N v^{\{k\}} = v_{0,1},
\end{eqnarray*}
where $v^{\{k\}}$ is the solution to
	 	\begin{eqnarray*}
\frac{\partial v^{\{k\}}}{\partial t'} &=& \lambda_{k} \frac{\partial^2 v^{\{k\}}}{\partial z_k^2}, \\
v^{\{k\}}(\cdot,0;t) &=& G(\cdot,t).
 	\end{eqnarray*}
Application of Theorem \ref{thm:BoundPDEDiff2}, with replacements $u\to v$, $g\to G$, $N\to N-r$, 
$\lambda_i \to \lambda_{i+r}$, 
$r\to 0$
gives the result.
\end{proof}



	\section{Error bounds for first and second order expansions} \label{sec:ErrorBounds}

	Here, we provide a proof of Theorem \ref{thm:BoundPDEDiff2} and its second order extension, Theorem \ref{thm:BoundPDEDiff3}.
	Specifically, we derive error bounds in terms of the eigenvalues $\lambda_i$, time $t$, and derivatives of the initial condition. 
	We first state some auxiliary results.
	
	\subsection{Equations for the expansion error}
	We start by formulating PDEs for the expansion error, which will allow us to use standard PDE estimates to bound error terms.

	 	\begin{lem}
		\label{lem:PDEerror}
		The expansion error $\hat{u}^\xi$ from (\ref{eq:PDEDef5},\ref{eq:expan-err}) is the solution of 
	    \begin{eqnarray}\label{eq:PDEuhatxi1}
	        \frac{\partial}{\partial t} \hat{u}^\xi &=& \mathcal{L}^{\nu_\xi} \hat{u}^\xi + f 
		 \quad \text{in } \mathbb{R}^N \times (0,T), \\ 
			\hat{u}^\xi(\cdot,0) &=& \left(\sum_{(w,\nu)\in\xi}w - 1\right) g \quad \text{in } \mathbb{R}^N, 
	    \end{eqnarray}
		with source term
	    \begin{eqnarray*}
	        f &=& \sum_{(w,\nu)\in\xi}w\left[ \mathcal{L}^{\nu} - \mathcal{L}^{\nu_\xi} \right]u^\nu + 
	         \left[\mathcal{L}^{\nu_\xi}-\mathcal{L}\right] u,
	    \end{eqnarray*}
        where $\nu_\xi \subseteq \{1,\ldots,N\}$ arbitrary. 
	\end{lem}
	\begin{proof}
	See Appendix \ref{app:PDEerror}.
	\end{proof}
	
Although Lemma \ref{lem:PDEerror} holds for all index sets $\nu_\xi$, we have in mind $\nu_\xi = \bigcap_{(w,\nu)\in\xi}\nu$
for which we will apply it later.
We will use the following simple result repeatedly.
\begin{lem} \label{rhs-lemma}
Let $f \in C^b$ and let $u$ be a classical 
solution to $\partial u\!\,/\!\,\partial t - \mathcal{L} u = f$, $u(\cdot,0)=0$, where $\mathcal{L}$ is any second order linear elliptic operator (possibly degenerate), with no zero-order terms, and $|f|\le C t^p$
everywhere, for some constant $C>0$ and $p\ge 0$. Then
\begin{eqnarray*}
|u| \le \frac{C}{p+1} t^{p+1}.
\end{eqnarray*}
\end{lem}
\begin{proof}
The right-hand side above is a super-solution.
\end{proof}

	\subsection{Error bounds for a first order expansion} \label{sec:FirstOrderFirstREigenvalues}
	
	Consider the first order approximation $u_{r,1}$ from (\ref{eqn:truncerr}), which has the expansion $\xi$ from Example
	\ref{ex-first-order}. 
    We will now derive bounds via the PDE for the error itself, as introduced in Section \ref{sec:Expansion}.	
	\begin{lem}
	\label{lem:PDEerror1}
		For the expansion given in equation (\ref{xi-first}), the expansion error $\hat{u}^\xi$ from (\ref{eq:expan-err})
		is the solution of 
	    \begin{eqnarray}\label{eq:PDEuhatxi3}
	        \frac{\partial}{\partial t} \hat{u}^\xi &=& \sum_{k=1}^r{ \lambda_k\frac{\partial^2}{\partial z_k^2}}\hat{u}^\xi + f
	        \quad \text{in } \mathbb{R}^N \times (0,T), \\
			\hat{u}^\xi(\cdot,0) &=& 0 \quad \text{in } \mathbb{R}^N,
			\label{eq:PDEuhatxi3Boundary}
	    \end{eqnarray}
		with source term
	    \begin{eqnarray}\label{eq:PDEuhatxi3Source}
	        f &=& \sum_{k=r+1}^N{ \lambda_k\frac{\partial^2}{\partial z_k^2} \left[ u^{\{1,\ldots,r,k\}}-u \right]}.
	    \end{eqnarray}
	\end{lem}
	\begin{proof}
See Appendix \ref{app:PDEerror1}.
	\end{proof}
	
We can now give a proof of Theorem \ref{thm:BoundPDEDiff2}.
    
	\begin{proof}[Proof (of Theorem \ref{thm:BoundPDEDiff2}).]
		Set $\hat{u}^{\{1,\ldots,r,k\}} = u^{\{1,\ldots,r,k\}}-u$. Then, taking the difference between (\ref{eq:PDEHeat}) and (\ref{eq:PDEDef4}), 
	    \begin{eqnarray}
			\frac{\partial}{\partial t}\hat{u}^{\{1,\ldots,r,k\}} &=& \sum_{i\in\{1,\ldots,r,k\}}\lambda_i\frac{\partial^2}{\partial z_i^2}u^{\{1,\ldots,r,k\}} - \sum_{i=1}^N\lambda_i\frac{\partial^2}{\partial z_i^2}u \nonumber\\
			&=& \sum_{i\in\{1,\ldots,r,k\}}\lambda_i\frac{\partial^2}{\partial z_i^2}u^{\{1,\ldots,r,k\}} - \sum_{i \in \{1,\ldots, r, k\}} \lambda_i\frac{\partial^2}{\partial z_i^2}u - \sum_{i \notin \{1,\ldots, r, k\}}\lambda_i\frac{\partial^2}{\partial z_i^2}u \nonumber\\
			&=& \sum_{i\in\{1,\ldots,r,k\}}\lambda_i\frac{\partial^2}{\partial z_i^2} ({u}^{\{1,\ldots,r,k\}}-u)  - 
			 \sum_{\scriptsize \begin{array}{c} i\!=\!r\!+\!1 \\i \!\neq\! k\end{array}}^N
			\lambda_i\frac{\partial^2}{\partial z_i^2}u, \nonumber \\
			&=& \sum_{i\in\{1,\ldots,r,k\}}\lambda_i\frac{\partial^2}{\partial z_i^2}\hat{u}^{\{1,\ldots,r,k\}} - 
			 \sum_{\scriptsize \begin{array}{c} i\!=\!r\!+\!1 \\i \!\neq\! k\end{array}}^N
			\lambda_i\frac{\partial^2}{\partial z_i^2}u, \label{laststep}
	    \end{eqnarray}
		for all $(z,t)\in\mathbb{R}^N\times(0,T]$, with zero initial condition.
Gleaning at (\ref{eq:PDEuhatxi3Source}), we need the second $k$-derivative of
$\hat{u}^{\{1,\ldots,r,k\}}$, but this is itself the solution to the inhomogeneous heat equation
	    \begin{eqnarray}
			\frac{\partial}{\partial t}\left( \frac{\partial^2}{\partial z_k^2} \hat{u}^{\{1,\ldots,r,k\}} 
			\right)
			&=& \sum_{i\in\{1,\ldots,r,k\}}\lambda_i\frac{\partial^2}{\partial z_i^2} \left( \frac{\partial^2}{\partial z_k^2} \hat{u}^{\{1,\ldots,r,k\}}\right) - 
			 \sum_{\scriptsize \begin{array}{c} i\!=\!r\!+\!1 \\i \!\neq\! k\end{array}}^N
			\lambda_i\frac{\partial^4}{\partial z_k^2 \partial z_i^2}u, \label{newstep}
	    \end{eqnarray}
as is seen by differentiating (\ref{laststep}) twice each with respect to $z_i$ and $z_k$.
In turn, by differentiation of (\ref{eq:PDEHeat}) twice with respect to $z_k$, each of the mixed $u$ derivatives on the right-hand side of
(\ref{newstep}) satisfies a heat equation with zero right-hand
side and initial condition
\[
\frac{\partial^4 g}{\partial z_k^2\partial z_i^2},
\]
so that the norm of the solution is bounded by the norm of this initial condition (this follows directly from the Green's function
representation (\ref{eq:PDEHeatSolutionNu}) of the solution to this PDE
or the weak maximum principle for parabolic PDEs, see e.g.\ Theorem 8 in Chapter 7 of \cite{evans}).
Applying Lemma \ref{rhs-lemma} twice gives the result.

	\end{proof}


	\subsection{Error bounds for a second order expansion}
	\label{subsec:second_order}
    
We now consider the second order expansion, $u_{r,2}$ from (\ref{eqn:truncexp}).
    
    \begin{ex}\label{pro:SecondOrderExp}
    The second order expansion has the form $u_{r,2}=u^\xi$, where
    \begin{eqnarray}\label{eq:xiTaylorExpansion2,r}
    	\xi &=& \{ (1 + (N-r)(N-r-3)/2,\{ 1,\ldots,r\}),\nonumber\\
        && (2-(N-r),\{ 1,\ldots,r,r+1\}), \ldots, (2-(N-r),\{ 1,\ldots,r,N\}),  \nonumber\\
        && (1,\{ 1,\ldots,r,r+1,r+2\}),(1,\{ 1,\ldots,r,r+1,r+3\}), \ldots, \nonumber\\
        && (1,\{ 1,\ldots,r,N-2,N\}),(1,\{ 1,\ldots,r,N-1,N\}) \}.
	\end{eqnarray}
    \end{ex}
	
    \begin{rem}
    Computing $u^\xi$ according to (\ref{eq:xiTaylorExpansion2,r}) requires the computation of one $r$-dimensional PDE, $N-r$ of $(r+1)$-dimensional PDEs and $(N-r)(N-r-1)/2$ of $(r+2)$-dimensional PDEs. Compared to the first order scheme, this adds one dimension to the highest dimensional PDEs and increases their number by a factor $O(N-r)$. In return, the approximation order increases by one.
    \end{rem}

    	\begin{lem}
	\label{lem:BoundPDEDiff3}
		For the expansion given in equation (\ref{eq:xiTaylorExpansion2,r}), the expansion error $\hat{u}^\xi$ from (\ref{eq:expan-err}) 
		is a solution of  (\ref{eq:PDEuhatxi3}--\ref{eq:PDEuhatxi3Boundary}),
	but now with source term
	    \begin{eqnarray}\label{eq:PDEuhatxi4Source}
	        f &=& - \sum_{k=r+1}^N \sum_{\scriptsize \begin{array}{c} l\!=\!r\!+\!1 \\l \!\neq\! k\end{array}}^N \sum_{\scriptsize \begin{array}{c} i\!=\!r\!+\!1 \\i\!\neq\! k,\!l\end{array}}^N
	        {{\lambda_i \lambda_k \frac{\partial^2}{\partial z_i^2} \frac{\partial^2}{\partial z_k^2} \left[ \tilde{u}^{\{1,\ldots,r,l,k\}}-\tilde{u}^{\{1,\ldots,r,k\}} \right]}},
	    \end{eqnarray}
	    where $\tilde{u}^\nu$ is the solution to
	     \begin{eqnarray}\label{eq:auxPDE}
	    \frac{\partial \tilde{u}^\nu}{\partial t} - \mathcal{L}^\nu \tilde{u}^\nu = u, \qquad \tilde{u}^\nu(\cdot,0) = 0.
	    \end{eqnarray}
	\end{lem}

	\begin{proof}
	See Appendix \ref{app:BoundPDEDiff3}.
	\end{proof}
    
    \begin{thm}\label{thm:BoundPDEDiff3}
    Assume $g\in C^{2,3,mix}$ in (\ref{eq:PDEHeat}--\ref{eq:PDEHeatBoundary}). Then the expansion error $\hat{u}_{r,2}$ satisfies
        \begin{eqnarray}
         \label{eq:BoundPDEDiff3}
			\left\| \hat{u}_{r,2}(\cdot,t) \right\|_{\infty} &\leq& t^3 \sum_{r<i<j<k\le N} \lambda_i \lambda_j \lambda_k \left\| \frac{\partial^6 g}{\partial z_i^2\partial z_j^2 \partial z_k^2} \right\|_\infty.
		\end{eqnarray}
	\end{thm}
    \begin{proof}
Set, for notational brevity,
\[
v: = \lambda_i \lambda_k \frac{\partial^2}{\partial z_i^2} \frac{\partial^2}{\partial z_k^2} \tilde{u}^{\{1,\ldots,r,k\}},  \qquad
w: = \lambda_i \lambda_k \frac{\partial^2}{\partial z_i^2} \frac{\partial^2}{\partial z_k^2} \tilde{u}^{\{1,\ldots,r,l,k\}},
\]
such that $\tilde{v}: = w-v$ is a single term in the sum on the right-hand side of (\ref{eq:PDEuhatxi4Source}).
Then differentiating (\ref{eq:auxPDE}) with respect to $z_i$ and $z_k$ twice, for $\nu=\{1,\ldots, r, k\}$ and $\nu=\{1,\ldots, r, k, l\}$, respectively, 
$v$ and $w$ are seen to satisfy
\begin{eqnarray}
\label{v-pde}
\frac{\partial v}{\partial t} - 
\mathcal{L}^{\{1,\ldots,r,k\}} v
&=& \lambda_i \lambda_k \frac{\partial^2}{\partial z_i^2} \frac{\partial^2}{\partial z_k^2} u, \\
\frac{\partial w}{\partial t} - 
\mathcal{L}^{\{1,\ldots,r,k,l\}} w
&=& \lambda_i \lambda_k \frac{\partial^2}{\partial z_i^2} \frac{\partial^2}{\partial z_k^2} u.
\label{w-pde}
\end{eqnarray}
Taking the difference between (\ref{v-pde}) and (\ref{w-pde}) gives an equation for 
$\tilde{v}$,
\begin{eqnarray}
\label{eqn:aux1}
\frac{\partial \tilde{v}}{\partial t} - 
\mathcal{L}^{\{1,\ldots,r,k,l\}}  \tilde{v}
&=& \lambda_l \frac{\partial^2 v}{\partial z_l^2}.
\end{eqnarray}
By differentiation of (\ref{v-pde}) twice with respect to $z_l$, it is also seen that
\begin{eqnarray}
\label{eqn:aux2}
\frac{\partial}{\partial t}\left(\lambda_l \frac{\partial^2 v}{\partial z_l^2}\right) - 
\mathcal{L}^{\{1,\ldots,r,k\}}
\left(\lambda_l \frac{\partial^2 v}{\partial z_l^2}\right)
&=& \lambda_l \lambda_i \lambda_k \frac{\partial^2}{\partial z_l^2} \frac{\partial^2}{\partial z_i^2} \frac{\partial^2}{\partial z_k^2} u.
\end{eqnarray}
Hence we deduce in turn, using Lemma \ref{rhs-lemma}, 
\begin{eqnarray*}
\left\|\lambda_l \frac{\partial^2 v}{\partial z_l^2}\right\|_\infty &\le& \; t \; \, \lambda_l \lambda_i \lambda_k
\left\|\frac{\partial^2}{\partial z_l^2} \frac{\partial^2}{\partial z_i^2} \frac{\partial^2}{\partial z_k^2} u \right\|_\infty
\hspace{5 cm} \text{(from (\ref{eqn:aux2}))},
\\
\left\| \tilde{v} \right\|_\infty &\le& \frac{t^2}{2} \, \lambda_l \lambda_i \lambda_k
\left\|\frac{\partial^2}{\partial z_l^2} \frac{\partial^2}{\partial z_i^2} \frac{\partial^2}{\partial z_k^2} u \right\|_\infty
\hspace{5 cm} \text{(from (\ref{eqn:aux1}))},
\\
\left\| \hat{u}^{\xi} \right\|_\infty &\le& \frac{t^3}{6} \, 
\sum_{k=r+1}^N \sum_{\scriptsize \begin{array}{c} l\!=\!r\!+\!1 \\l \!\neq\! k\end{array}}^N \sum_{\scriptsize \begin{array}{c} i\!=\!r\!+\!1 \\i\!\neq\! k,\!l\end{array}}^N
\lambda_l \lambda_i \lambda_k
\left\|\frac{\partial^2}{\partial z_l^2} \frac{\partial^2}{\partial z_i^2} \frac{\partial^2}{\partial z_k^2} u \right\|_\infty
\qquad \qquad \text{(from Lemma \ref{lem:BoundPDEDiff3})}.
\end{eqnarray*}
    \end{proof}

    \begin{rem}
   By extension of Remark \ref{rem:superposition1}, 
   there are no univariate or bivariate quadratic terms with factors $\lambda_k^2$ or $\lambda_j^2 \lambda_k^2$ in (\ref{eq:BoundPDEDiff3}). 
    Hence, any solutions which only depend on one or two of the $z_k$ are integrated exactly, and by superposition any linear combination of such terms. By similar reasoning to before, $\widehat{u}_{r,2}$ is zero for initial conditions of the form
    \begin{eqnarray*}
    g(x) = \sum_{j,k=r+1}^N
    g_{j,k}(x_j,x_k;x_1,\ldots,x_r) \;\; + \!\!
     \sum_{i,j,k=r+1}^N 
     x_i \,  g_{i,j,k}(x_j,x_k;x_1,\ldots,x_r),
    \end{eqnarray*}
for any functions $g_{j,k}$ and $g_{i,j,k}$.
    \end{rem}
    
    \begin{rem}
    One can again weaken the smoothness requirements on $g$, as discussed in Corollary \ref{cor:BoundPDEDiff2} for the first order case.
    \end{rem}

\begin{rem}
Comparing how (\ref{eq:BoundPDEDiff3}) emerges from (\ref{eq:PDEuhatxi4Source}) to how (\ref{eq:BoundPDEDiff2_1}) emerges from (\ref{eq:PDEuhatxi3}),
we conjecture  (but do not prove) for 
a higher order expansions of the form
(\ref{eqn:truncexp})
that
\begin{eqnarray*}
\|u - u_{r,m}\|_\infty &\le & t^{m+1} \sum_{\scriptsize \begin{array}{c} |\alpha|\!=\!m\!+\!1 \\ \alpha_i \!\in\! \{0,1\}\end{array}}
(\lambda-\lambda^0)^\alpha \|D^{2\alpha} g\|_{\infty}.
\end{eqnarray*}
\end{rem}

\section{Analysis and construction of approximations by Taylor expansion in $\lambda$}
\label{subsec:taylor}

This section provides an alternative derivation of the expansions.  This establishes a link between the ANOVA view and the literature motivated by Taylor expansions, \cite{RW07, HKSW10}. It will show that higher order approximation of the Taylor terms is not always advantageous (Section \ref{subsec:alternative}) and allows us to show sharpness of our previous results (Section \ref{subsec:sharpness}). 

		Let $u(z,t,\cdot)\in C^{m}$, i.e., $m$ times differentiable in $\lambda$, and fix $\lambda^0 \in \mathbb{R}^N$. Using multi-index notation, we write $D_\lambda^\alpha u = \frac{\partial^{|\alpha|} u}{\partial\lambda^{\alpha}} = \frac{\partial^{|\alpha|} u}{\partial\lambda_1^{\alpha_1}\ldots\partial\lambda_N^{\alpha_N}}$ and $(\lambda-\lambda^0)^\alpha = (\lambda_1 - \lambda^0_1)^{\alpha_1}\cdot \ldots\cdot (\lambda_N - \lambda_{0,N})^{\alpha_N}$. Then,
		by standard multivariate calculus (see, e.g., \cite{K04}),
	\begin{eqnarray}\label{eq:TaylorExpansion}
		u(z,t,\lambda) &=& \sum_{k=0}^m{\sum_{|\alpha|=k}{ \frac{(\lambda-\lambda^0)^\alpha}{\alpha !} D_\lambda^\alpha u (z,t,\lambda^0)}} + R_m(z,t,\lambda,\lambda^0,u),
	\end{eqnarray}
with a remainder term $R_m$. If $u(z,t,\cdot)\in C^{m+1}$, then an explicit form can be given as
	\begin{eqnarray}
\label{taylor-m}
		R_m(z,t,\lambda,\lambda^0,u) &=& \sum_{|\alpha|=m+1}{ (\lambda-\lambda^0)^\alpha R^\alpha(z,t,\lambda,\lambda^0,u)}, \\
		R^\alpha(z,t,\lambda,\lambda^0,u) &=& \frac{|\alpha|}{\alpha!}\int_0^1{ (1-s)^{|\alpha|-1} (D_\lambda^\alpha u)(z,t,\lambda^0+s(\lambda-\lambda^0)) ds} .
\label{taylor-alpha}
	\end{eqnarray}
	
	We will analyse the existence of these $\lambda$-derivatives in detail in Section \ref{subsec:existderiv}.

	\subsection{The first order case (see also Example \ref{ex-first-order} and Theorem \ref{thm:BoundPDEDiff2})}
		For $\alpha_k := {(0,\ldots,0,1,0,\ldots,0)}$, where the $k$-th entry is $1$, a first order finite difference approximation is given for $\delta \lambda_k>0$ by
	\begin{eqnarray}
		\frac{1}{\delta\lambda_k} \Delta^{\alpha_k} u(z,t,\lambda^0) = \frac{u (z,t,\lambda^0 + \delta\lambda_k e_k) - u (z,t,\lambda^0)}{\delta\lambda_k}, 
	\end{eqnarray}
	where $e_k$ is the $k$-th unit vector. Now choose $1\leq r\leq N$ and set $\lambda^0 = (\lambda_1,\ldots,\lambda_r,0,\ldots,0)$, $\delta\lambda = \lambda-\lambda^0$, and $m=1$. Furthermore, assume $u(z,t,\lambda)$ is twice continuously differentiable in $\lambda$. Then
	\begin{eqnarray}
	\nonumber
		u(z,t,\lambda) &=& D_\lambda^0 u (z,t,\lambda^0) + \sum_{|\alpha|=1}\frac{(\lambda-\lambda^0)^\alpha}{\alpha !} D_\lambda^\alpha u (z,t,\lambda^0) + R_1(z,t,\lambda,\lambda^0,u) \\ \nonumber
&=& u(z,t,\lambda^0)
	+ \sum_{k=r+1}^N\lambda_k \frac{u(z,t,\lambda^0+\lambda_k e_k) - u(z,t,\lambda^0) - 
	\lambda_k^2 R^{2\alpha_k}(z,t,\lambda^0+\lambda_k e_k,\lambda^0,u)}{\lambda_k} \\ 
&& + \sum_{k=r+1}^N{ \lambda_k^2 R^{2\alpha_k}(z,t,\lambda,\lambda^0,u)} 
+ \sum_{\scriptsize \begin{array}{c}k,\!l\!\!=\!\!r\!+\!1 \\ k\!\!\neq\!\! l\end{array}}^N{ \lambda_k\lambda_l R^{\alpha_k+\alpha_l}(z,t,\lambda,\lambda^0,u)}.
 \nonumber
\end{eqnarray}
Inserting the finite difference approximations for the first derivatives,
\begin{eqnarray}
\nonumber
u(z,t,\lambda)
	&=&
u_{r,1} +  \sum_{\scriptsize \begin{array}{c}k,\!l\!\!=\!\!r\!+\!1 \\ k\!\!\neq\!\! l\end{array}}^N{ \!\!\!  \!\! \lambda_k\lambda_l R^{\alpha_k+\alpha_l}(z,t,\lambda,\lambda^0,u)} \\
&& \qquad \!\! +  \sum_{k=r+1}^N   \lambda_k^2 \left(
R^{2\alpha_k}(z,t,\lambda,\lambda^0,u) - R^{2\alpha_k}(z,t,\lambda^0 \!+\!\lambda_k e_k,\lambda^0,u)\right).
\label{eq:TaylorExpansion1,r}
\end{eqnarray}
The order (in $\lambda$) of the second term in (\ref{eq:TaylorExpansion1,r}) agrees precisely with the results in Section \ref{sec:FirstOrderFirstREigenvalues}, and particularly with Theorem \ref{thm:BoundPDEDiff2}. We will analyse in Section \ref{subsec:existderiv} the relation between the regularity in $\lambda$
and the terms in (\ref{eq:BoundPDEDiff2_1}).
The last term is a higher order term for smooth (in $\lambda$) $R^{2\alpha_k}$, but it is this term which prevents us from directly deducing the compact form
of Theorem \ref{thm:BoundPDEDiff2} with only first order mixed terms.

\subsection{Alternative difference stencils in $\lambda$}
\label{subsec:alternative}

 The methods above use the standard one-sided finite difference approximation
 \[
\frac{ \Delta_k u(z,t,\lambda^0)}{\delta \lambda_k} = \frac{\partial}{\partial \lambda_k} u (z,t,\lambda^0) + O(\delta \lambda_k),	
\]
 and, for $\delta \lambda\ge 0$ and $\alpha \in \{0,1\}^N$ such that $\delta \lambda_k > 0$ if and only if $\alpha_k=1$, then
 \[
	\frac{\Delta^\alpha u(z,t,\lambda^0)}{\delta \lambda^\alpha}  = D_\lambda^\alpha u (z,t,\lambda^0) + \sum_{k=1}^N O(\alpha_k \delta \lambda_k).
\]
	
	Considering (\ref{eq:TaylorExpansion}), the question arises if there is any benefit in evaluating the partial derivatives $D^\alpha_\lambda$ 
 by higher order finite difference approximations $\mathcal{D}^\alpha$, such that
	\begin{equation}\label{eq:fd}
	 \mathcal{D}^\alpha u(z,t,\lambda^0) = D_\lambda^\alpha u (z,t,\lambda^0) + \sum_{\beta\in \mathcal{S}_\alpha}O(\delta\lambda^{\beta}),
	\end{equation}
	where $\mathcal{S}_\alpha$ is a set of multi-indices containing the orders of all error terms.
	So for $\Delta^\alpha$ from before, 
	$\mathcal{S}_\alpha = \{e_k: \alpha_k>0 \}$, $e_k$ the $k$th canonical vector.
	For example, \cite{HKSW10} proposes to use high order compact finite difference stencils introduced in \cite{L92}.

	Combining equations (\ref{eq:fd}) and (\ref{eq:TaylorExpansion}) gives
	\begin{eqnarray*}
		u(z,t,\lambda) &=& \sum_{k=0}^m\sum_{|\alpha|=k}\frac{(\lambda-\lambda^0)^\alpha}{\alpha !}\mathcal{D}^{\alpha}u(z,t,\lambda^0) \nonumber\\
		&& \;+ \; \sum_{k=1}^m\sum_{|\alpha|=k}\sum_{\beta\in \mathcal{S}_\alpha} O\left((\lambda-\lambda^0)^\alpha \delta\lambda^{\beta}\right)+ R_m(z,t,\lambda,\lambda^0,u).
	\end{eqnarray*}
	If we choose $\delta\lambda = \lambda-\lambda^0$ and assume $u \in C^{m+1}$ in $\lambda$, the leading order errors are of size
	\begin{eqnarray*}
	\sum_{|\alpha|=m+1} O\left((\lambda-\lambda^0)^{\alpha}\right) +
	    \sum_{k=1}^m\sum_{|\alpha|=k}\sum_{\beta\in \mathcal{S}_\alpha} O\left((\lambda-\lambda^0)^{\alpha+\beta}\right).
	\end{eqnarray*}
	The order of the first sum is determined by the expansion order (assuming sufficient smoothness) and will be the limiting factor to the overall order in practice.
	The order of the terms in the second sum can be increased by choosing higher order stencils,
	provided enough smoothness in $\lambda$, i.e., derivatives up to order $\alpha+\beta$ exist.
	This comes at the expense of additional computational complexity. 
	More specifically, computing $\mathcal{D}^\alpha u (z,t,\lambda^0)$ requires the values $u (z,t,\lambda')$ for different points $\lambda'$, depending on the finite difference scheme. 
	As the dimensionality of the PDEs is the same, the complexity increases by a constant multiplicative factor and therefore not crucially.
	All these schemes can be translated into an expansion using the $\nu/\xi$-notation introduced in Section \ref{sec:Expansion}. 
	
\begin{rem}
A disadvantage of these higher-order stencils is that additional (univariate) terms appear in the error (\ref{eq:BoundPDEDiff2_1}), and the exact computation of certain functions with low superposition dimension discussed in Remark \ref{rem:superposition1} is lost. A similar comment extends to the second-order case.
For the same reason, it does not seem advisable to try and increase the accuracy of the finite difference approximations by using a smaller step size than $\lambda_i-\lambda_{0,i}$, or to compute the derivatives more or less exactly by, say, algorithmic differentiation.
\end{rem}

\subsection{Sharpness of the bounds}
\label{subsec:sharpness}

Here, we will show that the error bounds (\ref{eq:BoundPDEDiff2_1}) and (\ref{eq:BoundPDEDiff3}) are `asymptotically sharp'
in the following sense:
We will give initial data $g$ such that
\begin{eqnarray}
\label{sharpness}
\|u - u_{r,m} \|_\infty &=& t^{m+1} 
\sum_{\scriptsize \begin{array}{c} |\alpha|\!=\!m\!+\!1 \\ \alpha_i \!\in\! \{0,1\}\end{array}}
(\lambda-\lambda^0)^\alpha \|D^{2 \alpha} g\|_{\infty} + o(t^{m+1} |\lambda-\lambda^0|^{m+1}).
\end{eqnarray}
Take without loss of generality $r=0$, $\lambda^0=0$ and
\[
g(z) = \prod_{k=1}^N \cos(z_k),
\]
such that $\|g\|_\infty = g(0) =1$ and $\|D^{2\alpha} g\|_\infty = |D^{2\alpha} g(0)| =1$ for all $\alpha$. 
The solution to (\ref{eq:PDEHeat}, \ref{eq:PDEHeatBoundary}) is then
\[
u(z,t) = \exp\left(- t \sum_{k=1}^N \lambda_k \right) \prod_{k=1}^N \cos(z_k).
\]
By (\ref{eq:TaylorExpansion}) one gets that
\begin{eqnarray*}
\exp\left(- t \sum_{k=1}^N \lambda_k \right) &=& 
\sum_{k=0}^\infty \frac{1}{k!}
 \left(-t \sum_{k=1}^N \lambda_k \right)^k 
\;=\;\; 1 + \sum_{k=1}^N \left(\e^{-t \lambda_k}-1 \right) 
+ \bar{R}_1 \\
&=& 1 + \sum_{k=1}^N \left(\e^{-t \lambda_k}-1 \right) + \sum_{1\le j< k}^N \left(\e^{-t \lambda_k-t\lambda_j}-\e^{-t \lambda_k}-\e^{-t \lambda_j}+1 \right)
+ \bar{R}_2
\end{eqnarray*}
etc, where 
\[
\bar{R}_m =
\sum_{\scriptsize \begin{array}{c} |\alpha|\!=\!m\!+\!1 \\ \alpha_i \!\in\! \{0,1\}\end{array}}
(\lambda-\lambda^0)^\alpha R^\alpha(z,t,\lambda,\lambda^0,u),
\]
i.e., $\bar{R}_m$ contains only the mixed terms of $R_m$ in (\ref{taylor-m}).
One calculates explicitly from (\ref{taylor-alpha}),
\begin{eqnarray*}
R^\alpha &=& 
 \frac{|\alpha|}{\alpha!}\int_0^1{ (1-s)^{|\alpha|-1} t^{|\alpha|} u(z,t,\lambda^0+s(\lambda-\lambda^0)) ds}
\;\;=\;\;  t^{|\alpha|} u(z,t,\lambda^0) + o(t^{|\alpha|}).
\end{eqnarray*}
From this the claim (\ref{sharpness}) follows.
In fact, we see that locally, in a neighbourhood of $(z,t) = (0,0)$, (\ref{sharpness}) will generally describe
the behaviour of the error for smooth $g$.
	

	\section{Regularity in $\lambda$ and $z$} 
	\label{subsec:existderiv}

	In this section, we establish conditions that guarantee the existence of partial derivatives 
	in $\lambda$, as needed for the Taylor expansions in Section \ref{subsec:taylor}.
	We also derive bounds on the size of those derivatives, so that an explicit upper bound for the remainder term in the Taylor expansion can be given, thus bridging
	the analyses from Sections \ref{sec:FirstOrderFirstREigenvalues} and \ref{subsec:second_order}
	on the one hand,
	and \ref{subsec:taylor} on the other. 
	
	In particular, we examine when $D_\lambda^\alpha u$ exists for non-smooth initial conditions. 
		It is a well-known phenomenon that the heat equation 
		has a smoothing effect on non-smooth initial data. This has been observed in \cite{GKS10} from an ANOVA decomposition perspective, where the authors are, as here, motivated by option pricing problems. Here, we analyze different types of payoffs in more detail.
Thus, we will be able to answer the regularity question positively for a wide class of practically relevant problems. This includes initial conditions which are not differentiable everywhere --- for example ones with kinks or jumps --- as long as the non-differentiabilities are localised. Examples \ref{ex:DigitalGeometricBasket} and \ref{ex:GeometricBasket} illustrate this effect.

%
%
\subsection{The smooth case}
\label{subsec:smooth}

	Let us first consider the existence and form of partial derivatives $D_\lambda^\alpha u$ for smooth $g$. 
\begin{pro}
\label{pro:smooth}
Let $u$ be a solution to (\ref{eq:PDEHeat})
with $g \in C^{2,1,mix}$ and $\Phi = \Phi^{\{1,\ldots,N\}}$, then for all $z\in \mathbb{R}^N$, $i=1,\ldots,N$,
\begin{eqnarray}
\label{eqn:firstlamder}
	\frac{\partial}{\partial\lambda_i} u(z,t,\lambda)
	&=& t \int_{\mathbb{R}^N} \Phi(y,t,\lambda) \, \frac{\partial^2 g}{\partial z_i^2}(z-y) \dy.
\end{eqnarray}
If, moreover, 
$g \in C^{2,2,mix}$, then for $i\neq j$
\begin{eqnarray}
\label{eqn:seclamder}
	\frac{\partial^2}{\partial\lambda_i \partial\lambda_j} u(z,t,\lambda)
	&=& t^2 \int_{\mathbb{R}^N} \Phi(y,t,\lambda) \,\frac{\partial^4 g}{\partial z_i^2 \partial z_j^2}(z-y) \dy.
\end{eqnarray}
\end{pro}
\begin{proof}
This follows directly by differentiating the Green's function representation (\ref{eq:PDEHeatSolutionNu}).
\end{proof}

An important point to note from (\ref{eqn:seclamder}) is that as $\lambda_i, \lambda_j \rightarrow 0$,
the integral tends to the second mixed derivative of $g$ at $z$. These are precisely the error terms found in (\ref{eq:BoundPDEDiff2_1}).

\subsection{The piecewise smooth case and data smoothing}
\label{subsec:nonsmooth}
 
The motivation for the research in this paper are applications in derivative pricing, where the initial data are typically non-smooth.
The explicit examples in Appendix \ref{subsec:exnonsmooth} show that this does not necessarily mean that the results from Section \ref{subsec:smooth} are not applicable for piecewise smooth data, in fact, with the exception of degenerate cases, which we will characterise in this section, the same convergence orders as in the smooth case hold.

        \begin{cor}
    [to Proposition \ref{pro:smooth}]
Let $u$ be a solution to (\ref{eq:PDEHeat}) and $G$ from (\ref{eq:DefG}).
Assume the conditions imposed  on $g$ in Proposition \ref{pro:smooth} on $\mathbb{R}^{N}$ are replaced by the same
conditions on $G(z_1,\ldots,z_r,\cdot,\ldots,\cdot,t)$ on $\mathbb{R}^{N-r}$, for fixed $z_1,\ldots,z_r$.
Then the result still holds
for all $i > j > r$ and all $\lambda\ge 0$ such that $\lambda_1, \ldots, \lambda_r >0$.
\end{cor}

	In the following, let $x_{-j} = (x_1,\ldots,x_{j-1},x_{j+1},\ldots,x_N)$ and let $dx_{-j}$ be shorthand for $dx_1\ldots dx_{j-1}dx_{j+1}\ldots dx_N$.
	We also write $\sqrt{\lambda} = (\sqrt{\lambda_1},\ldots,\sqrt{\lambda_N})$ and $x\cdot y = (x_1 y_1,\ldots,x_N y_N)$ for the element-wise product.
	
	
	
	\begin{pro}\label{lem:diffutilde}
	Let $g \in C^b$ be such that, for
		\begin{eqnarray*}
			G_j(z_{-j},z_j,t,\lambda) = \int_{\mathbb{R}^{N-1}} \Phi^{\{1,\ldots,N-1\}}(x_{-j},t,1) \, g(z-\sqrt{\lambda}\cdot (x_1,\ldots,x_{j-1},0,x_{j+1},\ldots,x_N) ) \dx_{-j},
		\end{eqnarray*}
		and every fixed $z_{-j},t,\lambda$, $G_j(z_{-j},\cdot,t,\lambda)$ is piecewise $C^2$, 
		i.e., there
		is a function $\widetilde{G}_j \in C^{2,1,mix}$, such that
		\begin{eqnarray}
		\label{G-decomp}
		G_j(z_j) =  \widetilde{G}_j(z_j) + \sum_{i=0}^N c_{i,j} \max(z_j-a_{i,j},0) + 
		 \sum_{i=0}^N b_{i,j} H(z_j-a_{i,j}),
		\end{eqnarray}
		where $H$ is the Heaviside function and $a_{i,j}, b_{i,j}, c_{i,j}$ given.
	
    Then the following holds:
    \begin{enumerate}
        \item $\frac{\partial}{\partial\lambda_j}u(z,t,\lambda)$ exists for all $z$ and all $\lambda$ with $\lambda_j > 0$;
        
		\item for a given $z_j$,
		\begin{eqnarray*}
			\left.\frac{\partial}{\partial\lambda_j} u(z,t,\lambda)\right|_{\lambda_j = 0} &=& t \, \left. \frac{\partial^2 G_j}{\partial z_j^2}(z_{-j},z_j,t,\lambda)
			\right|_{\lambda_j = 0},
			\end{eqnarray*}
			if the derivative 
			on the right-hand side exists.
    \end{enumerate}
	\end{pro}
	\begin{proof}
	This follows by studying the terms in (\ref{G-decomp}) individually. The smooth part is covered by Proposition \ref{pro:smooth},
	the piecewise linear Lipschitz and step function parts are as in Examples \ref{ex:GeometricBasket} and \ref{ex:DigitalGeometricBasket}
	 in Appendix \ref{subsec:exnonsmooth}.
	For a more explicit proof see \cite{W15}.
	\end{proof}
	
%


    Proposition \ref{lem:diffutilde} gives a sufficient but indirect condition for the existence of $\lambda$-derivatives in terms of $G_j$.
    The same degree of smoothness is needed for the error bounds in Section \ref{sec:ErrorBounds}.
    To understand what conditions on $g$ guarantee that $G_j$ satisfies the requirements, it is instructive to consider first the two-dimensional case. The key idea is that integration in one direction, $z_1$, helps smooth out discontinuities in $z_2$, as long these 
    are not orthogonal to $z_1$.

    \begin{ex}[Mixed smoothness]
    \label{ex:smoothing}
    Consider Figure \ref{fig:smoothing}. 
               \begin{figure}[tbp]
\begin{center}
\begin{minipage}[r]{0.9\textwidth}
\begin{center}
\begin{tikzpicture}[scale=0.9]
    \begin{axis}[ylabel=$z_2$, xlabel=$z_1$, xmin=0, xmax=10, ymin = 0,ymax=10]
        \addplot+[fill,color=lightgray,no markers] coordinates
            {(.5,.5) (.5,8) (9.5,8) (9.5,5) (5,5)} --cycle;

        \addplot+[only marks,mark options={color=black,fill=black},mark size=2,nodes near coords,point meta=explicit symbolic,every node near coord/.style={anchor=315}] coordinates                {
            (2,8) [$p_1$]
            (2,6.5) [$p_2$] 
            (2,5) [$p_3$]
            (2,3.5) [$p_4$] 
            (2,2) [$p_5$]
        };
	\end{axis}
\end{tikzpicture}
\end{center}
\end{minipage}
\end{center}
\caption{Schematic of the payoff function for Example \ref{ex:smoothing}, to illustrate how integration in $z_1$ affects the differentiability in $z_2$: $g\equiv 1$ inside the shaded area and $g\equiv 0$ everywhere else.}
\label{fig:smoothing}
\end{figure}
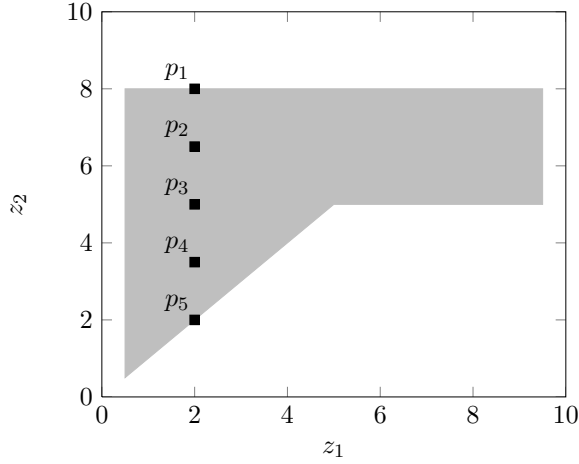
    Out of the five marked points, $g$ is twice continuously differentiable only in points $p_2$, $p_3$ and $p_4$. Once integrated over $z_1$, the resulting function $G_1$ is twice continuously differentiable in $p_2$, $p_4$ and $p_5$. Integration helps to smooth out the discontinuity in $p_5$. It does not help for $p_1$, where the kink is orthogonal to the $z_2$-direction and the discontinuity is orthogonal to $z_1$. Integration over $z_1$ decreases the smoothness in $p_3$, since $g$ is differentiable in $p_3$ but not differentiable on $\{ (z_1,z_2) : 5\leq z_1\leq 10, z_2 = 5\}$.
    \end{ex}

     \begin{ex}[Single kink or discontinuity]
    \label{ex:smoothing2}
    This is essentially the setting also studied in \cite{GKS10}.
    Assume now that $g : \mathbb{R}^2\rightarrow \mathbb{R}$ is twice continuously differentiable everywhere but on a 
    set
    \begin{eqnarray*}
    \{ (z_1,z_2) : f(z_1,z_2) = 0 \},
    \end{eqnarray*}
    given by the roots of a smooth function $f: \mathbb{R}^2\rightarrow \mathbb{R}$. For simplicity assume a single kink or discontinuity.
    More specifically, we assume that $\frac{\partial f}{\partial z_1}>0$ (or $<0$) everywhere.
    
    Then for every value of $z_2$ there is at most one value of $z_1$ such that $f(z_1,z_2) = 0$, and,
    using the implicit function theorem there exists a smooth function $b: \mathbb{R} \rightarrow \mathbb{R}$ describing the location of the kink or discontinuity,
    \begin{eqnarray*}
        f(b(z_2),z_2) = 0 \quad \forall z_2\in \mathbb{R}.
    \end{eqnarray*}
   
    \noindent
    Writing out the partial derivative now leads to
    \begin{eqnarray*}
    \frac{\partial}{\partial z_2}G_1(z,t,x_2,\lambda)
    &=& \frac{\partial}{\partial z_2} \int_{\mathbb{R}} \Phi(x_1,t,1) g(z-\sqrt{\lambda}\cdot x) dx_1\\
\hspace{-2 cm}    &\hspace{-2.2 cm}=&\hspace{-1.1 cm} \frac{\partial}{\partial z_2} \left[ \int_{-\infty}^{b(z_2)} \Phi(x_1,t,1) g(z-\sqrt{\lambda}\cdot x) dx_1 + \int_{b(z_2)}^{\infty} \Phi(x_1,t,1) g(z-\sqrt{\lambda}\cdot x) dx_1 \right]\\
\hspace{-2 cm}     &\hspace{-2.2 cm}=&\hspace{-1.1 cm} \frac{\partial b}{\partial z_2}(z_2) \Phi(b(z_2),t,1) \left[ g(z-\sqrt{\lambda}\cdot (b(z_2)^-,x_2)) - g(z-\sqrt{\lambda}\cdot (b(z_2)^+,x_2))\right]\\
    &&\hspace{-1.1 cm} \;+ \int_{-\infty}^{b(z_2)} \Phi(x_1,t,1) \frac{\partial g}{\partial z_2}(z-\sqrt{\lambda}\cdot x) dx_1 + \int_{b(z_2)}^{\infty} \Phi(x_1,t,1) \frac{\partial g}{\partial z_2}(z-\sqrt{\lambda}\cdot x) dx_1.
    \end{eqnarray*}
 Here, $b$ is a smooth function and $g$ is $C^2$ in $z_2$ everywhere but at the kink. Thus every term in the sum and the sum itself are continuously differentiable in $z_2$. This can also be seen by writing out the second derivative.
    
    
    It is straightforward to
    generalise this argument to finitely many kink points per $z_2$ by observing that $f$ is then locally invertible. 
    
    \end{ex}

In higher dimensions, for $z_j$-derivatives of $G$ to exist in a point $z^0 = (z_1^0,\ldots, z_N^0)$, it is sufficient that either 
\begin{itemize}
\item
$g$ is sufficiently $z_j$-differentiable for all points in a plane through $z^0$ orthogonal to the $z_j$-axis, or,
\item
smoothing of any discontinuities occurs through integration over $z_1,\ldots, z_r$, which happens if the location of the discontinuities is not parallel to all $z_1,\ldots, z_r$. This is guaranteed if the set is locally described by 
\[
f(z_1,\ldots,z_r; z_{r+1}^0,\ldots,z_{j+1}^0,z_j,z_{j+1}^0,\ldots, z_N^0)=0,
\]
for fixed
$z_{r+1}^0,\ldots,z_{j-1}^0,z_{j+1}^0,\ldots, z_N^0$, with smooth $f$, and
\[
\frac{\partial}{\partial z_k}f(z_1,\ldots,z_r; z_{r+1}^0,\ldots,z_{j+1}^0,z_j,z_{j+1}^0,\ldots, z_N^0) \neq 0
\]
for at least one $1\le k\le r$, i.e., the kink set is not parallel to the plane and smoothing occurs across it in at least one variable $z_k$.
\end{itemize}

\section{Numerical computations}\label{sec:Numerics}

	\subsection{Definition of the model problem}
	\label{subsec:defmodel}
	
	Consider the Black-Scholes model 
	with stock price processes $S_1,\ldots,S_N$ following
	\begin{eqnarray}
		\dS_{i,t} = r_{f} S_{i,t} \dt + \sigma_i S_{i,t} \dW_{i,t} \;\;\; 1\leq i \leq N,
		\label{blackscholes}
	\end{eqnarray}
	where $r_{f}\in \mathbb{R}$ is the risk-free interest rate and $\sigma \in \mathbb{R}^N$ are
the constant volatilities. For the standard $N$-dimensional Brownian motion $W$ 
we take a correlation matrix 
	\begin{eqnarray}
	\label{eqn:corr}
		\rho = \left( \begin{array}{ccccc}
		1 & \gamma & \gamma & \cdots & \gamma\\
		\gamma & 1 & \gamma & \cdots & \gamma\\
		\vdots & & \ddots & \vdots\\
		\gamma & \gamma & \gamma & \cdots & 1
		\end{array} \right),
	\end{eqnarray}
	for some $\gamma \in (-1,1)$.
	The covariance matrix $\Sigma \in \mathbb{R}^{N\times N}$ is given by $\Sigma_{ij} = \sigma_i\sigma_j\rho_{ij}$. 
	In the simplest case of identical volatilities $\sigma_1=\ldots = \sigma_N = \sigma$,
this implies that there are at most two distinct eigenvalues $\lambda_1\geq \lambda_2$ (with $\lambda_1 = \lambda_2$ if and only if $\gamma=0$) such that $\lambda=(\lambda_1,\lambda_2,\lambda_2,\ldots,\lambda_2)$. The normalized eigenvector corresponding to $\lambda_1$ is $q_1= (1,\ldots,1)/\sqrt{N}$ and the other eigenvectors can be chosen as any orthogonal basis of the orthogonal complement of $\{q_1\}$.
	Writing the covariance matrix as $\sigma^2 (\gamma E + (1-\gamma) I)$, where $I$ is the $N\times N$ identity matrix and $E$ the matrix full of ones, and noting that $E q_1 = N q_1$ and $E q_2 = 0$, one finds the eigenvalues $\lambda_1 = \sigma^2 ((N-1) \gamma + 1)$ and $\lambda_2 = \sigma^2 (1- \gamma)$.
	In particular, $\lambda_2$ and the higher eigenvalues do not depend on $N$. 
	
	This covariance matrix is clearly stylised, however, numerical tests with more general models confirm that the dominant effect comes from the spectral gap after the $r$th (here, first) eigenvalue, whereas the more detailed structure of the correlations is less relevant.
For instance, \cite{RW13} uses a correlation matrix with elements $\rho_{ij} = \exp(-\alpha |i-j|)$ for some $\alpha>0$ in the context of the LIBOR market model, and analyses numerically the impact of $\gamma$ on the eigenvalues and the numerical accuracy of truncated expansion approximations to Bermudan swaptions with $N=5,11,21,41,\cdots,101$ tenors.
Closer to the test cases here, \cite{RW07} gives results for arithmetic basket options, but with an empirically estimated covariance matrix for the DAX, i.e.\ $N=30$. 
There, the eigenvalues jump by roughly a factor of 10 after the first, and then decay slowly but exponentially. A similar size jump is observed in Figure \ref{fig:varlambda2} for $\gamma \approx 0.4$ to $0.5$, a typical correlation between equity prices.
The accuracy for the example in \cite{RW07} is $0.73$ basis points.

	For simplicity, we choose $r_{f} = 0$. The arbitrage-free price $V(S_0,0)$ at time $0$,
of a financial derivative with payoff $h(S_{T})$ at time $T$ 
is then given by
    \begin{eqnarray}
        V(S_0,0) = \mathbb{E}\left[\; h(S_T) \; | \; S_0 \;\right]\label{eq:FairPrice}.
    \end{eqnarray}
    We can estimate this value with desired accuracy using, for example, Monte Carlo simulation with a sufficiently large number of samples. This provides a reference to compare the PDE solution to.
    
    The expectation (\ref{eq:FairPrice}) is given by the solution $u(\log(S_0),T)$ of
    \begin{eqnarray*}
        \frac{\partial u}{\partial t} &=& \sum_{i,j=1}^N  \sigma_i\sigma_j\rho_{ij} \frac{\partial^2 u}{\partial x_i \partial x_j} - \sum_{i=1}^N \frac{\sigma_i^2}{2} \frac{\partial u}{\partial x_i}, \\
        u(x,0) &=& g(\exp(x)),
    \end{eqnarray*}
    where we have used a similar (logarithmic) transformation as in Appendix \ref{subsec:exnonsmooth}.

For constant, positive semidefinite
coefficient matrix $\Sigma = (\sigma_i \sigma_j \rho_{ij})_{1\leq i,j\leq N}$,
	we can choose $Q$ to be the matrix of eigenvectors of $\Sigma$ sorted by eigenvalue size, i.e.,
    \begin{eqnarray}
    Q = (q_1,\ldots,q_N),\;\; \frac{1}{2} \Sigma q_i = \lambda_i q_i,\;\; \lambda_1 \geq \ldots\geq\lambda_N \geq 0, 
    \end{eqnarray}
    and get $\Lambda = \diag(\lambda_1,\ldots,\lambda_N)$ a constant diagonal matrix.

Introducing the new coordinates
\[
z(x,t) = Q^T (x + \mu t),
\]
with $\mu = (-\sigma_i^2/2)_{i=1,\ldots,N}$ and inverse transform $x(z,t) = Q z - \mu t$,
leads to the heat equation (\ref{eq:PDEHeat}, \ref{eq:PDEHeatBoundary}) with $g(z)=h(x(z,0))$.

\subsection{Numerical approximation of low-order terms}
\label{subsec:numericalPDEs}

For each low-dimensional term in the decomposition, a PDE on an unbounded spatial domain is solved.\footnote{In the present case of constant coefficients, semi-analytic or Fourier methods would be even more efficient to solve these sub-problems.} To avoid the introduction of artificial boundary conditions necessary when localising the domain, for each coordinate $z_i$ we map the interval $(-\infty,\infty)$ to $(0,1)$ via
\begin{align*}
	y_i = \frac{1}{\pi}\arctan\left(b_i z_i + c_i\right)+ \frac{1}{2},
\end{align*}
with parameters $b_i$ and $c_i$.

Under a standard growth condition on the solution at infinity, the resulting PDE is fully specified without boundary conditions at $z_i \in \{0,1\}$, because the resulting non-constant coefficients of the transformed diffusion equation vanish sufficiently fast at the boundaries (see \cite{RW07,ZL03}). For options with unbounded payoffs towards infinity, such as European calls, we introduce an artificial numerical cutoff. This does not impact the computed option value noticeably.

We consider an equidistant mesh in the $y_i$ coordinates, such that in original coordinates the mesh is denser in the most relevant regions. This density can be adjusted by choosing the parameters $b_i$ and $c_i$, which is performed heuristically.

We use the Crank-Nicolson time discretization scheme with central spatial differences. In the one-dimensional case, solving the PDE is then straightforward. In the two- and three-dimensional case, we combine this with an Alternating Direction Implicit (ADI) factorisation as detailed below, such that the resulting tridiagonal matrix systems can be solved efficiently in linear time (i.e., proportional to the system size). An initial LU factorisation of the tridiagonal matrices gives significant further speed-up.

	\subsection{Variance reduction of Monte Carlo estimators}
	\label{subsec:varred}
	
We consider now a simulation-based method for solving the PDE (\ref{eq:PDEHeat}, \ref{eq:PDEHeatBoundary}), making use of the
expansion method as outlined in Section \ref{sec:ExpansionMethod}. We will use this for numerical verification only, but envisage this to be a good standalone variance reduction method more generally.

We recall from (\ref{anova1}) the decomposition
\begin{eqnarray}
u(z,0,\lambda) &=& \!\!\! \sum_{\alpha \in \{0,1\}^N} 
\underbrace{\delta\lambda^\alpha \Delta^\alpha u(z,0,\lambda^0)}_{=:u_\alpha}.
\label{def:ualpha}
\end{eqnarray}
We consider independent estimators $\widehat{U}_\alpha$ for $u_\alpha$ in (\ref{def:ualpha}), and an estimator
\begin{eqnarray*}
\widehat{U} &=& \!\!\! \sum_{\alpha \in \{0,1\}^N}  \widehat{U}_\alpha
\end{eqnarray*}	
for $u$. If the single sample variance of $\widehat{U}_\alpha$ is $v_\alpha$, then, using $N_\alpha$ independent samples for $\widehat{U}_\alpha$,
\begin{eqnarray*}
Var\left[ \widehat{U} \right] &=& \!\!\! \sum_{\alpha \in \{0,1\}^N} \frac{v_\alpha}{N_\alpha},
\end{eqnarray*}	
using independent samples for different $\alpha$.
We expect from the theoretical results in Section \ref{sec:ErrorBounds}, given enough smoothness,
$v_\alpha = O(\lambda_k^2)$ for the first order corrections, and 
$v_\alpha = O(\lambda_k^2 \lambda_l^2)$ for the second order corrections. 

    
    To achieve an overall mean-square error of size $\epsilon$ for the first order terms, one thus needs $O(\lambda_k^2 \epsilon^{-2})$ MC samples, instead of the $O(\epsilon^{-2})$ required for the 
    naive 
    approach of simulating $\widehat{U}$ directly with a single set of samples. In the cases below, $\lambda_k$ was often in the order of $10^{-2}$, which resulted in a significant run time reduction by a factor of about $10, 000$. 


Estimating the first order approximation $u_{1,1}$ in that way, requires $N$ individual estimators. Computing the whole expansion in this way would require $2^N$ estimators and therefore would be infeasible for large $N$, and would stop to be advantageous already for moderate $N$. A better approach is to compute an accurate approximation to $u_{1,1}$ using finite difference methods, and then use a Monte Carlo estimator
$\widehat{U}_{1,1}$ for $\hat{u}_{1,1}$, the difference  between $u$ and $u_{1,1}$, with a much smaller variance than estimating $u$ directly,
and similar for higher order. This is the approach taken in the computations in Sections \ref{subsec:numfirst} and \ref{subsec:numsecond}
for the smallest values of $\lambda$.

	\subsection{Convergence of the first order expansion}
    \label{subsec:numfirst}
    
    \begin{figure}[tbp]
    \begin{minipage}[r]{0.3\textwidth}
\begin{tikzpicture}[scale=.6]
    \begin{axis}[ylabel=$\Delta$, xlabel=$\lambda_2$, xmode=log, ymode=log, ymin=0.0000001, ymax=1, xmin=.0001, xmax=.1, title={$\omega_1$}]
       \addplot+[only marks,mark=-,mark options={color=black,fill=black},mark size=3] coordinates {
(0.024,0.0170014434164)
(0.02,0.0107028157909)
(0.016,0.0063269810485)
(0.012,0.00294553740244)
(0.008,0.00130794540049)
(0.004,0.000355126467481)
(0.002,8.26483784071e-05)
(0.0008,1.29292160655e-05)
(0.0004,3.40717369403e-06)
};
\addplot+[only marks,mark=x,mark options={color=black,fill=black},mark size=3] coordinates {
(0.024,0.0167438058213)
(0.02,0.0104190432586)
(0.016,0.00601877181581)
(0.012,0.00261425013254)
(0.008,0.000954682068198)
(0.004,0.000354515981159)
(0.002,8.24332533762e-05)
(0.0008,1.28749875969e-05)
(0.0004,3.38823042156e-06)
};
\addplot+[only marks,mark=-,mark options={color=black,fill=black},mark size=3] coordinates {
(0.024,0.0164861682262)
(0.02,0.0101352707264)
(0.016,0.00571056258313)
(0.012,0.00228296286264)
(0.008,0.0006014187359)
(0.004,0.000353905494837)
(0.002,8.22181283453e-05)
(0.0008,1.28207591284e-05)
(0.0004,3.36928714909e-06)
};


	\end{axis}
\end{tikzpicture}
\end{minipage}
\begin{minipage}[c]{0.01\textwidth}
$\,$
\end{minipage}
\begin{minipage}[l]{0.3\textwidth}
\begin{tikzpicture}[scale=.6]
    \begin{axis}[ylabel=$\Delta$, xlabel=$\lambda_2$, xmode=log, ymode=log, ymin=0.0000001, ymax=1, xmin=.0001, xmax=.1, title={$\omega_2$}]
		\addplot+[only marks,mark=-,mark options={color=black,fill=black},mark size=3] coordinates {
(0.024,0.0170035121579)
(0.02,0.0107044005044)
(0.016,0.00632816946199)
(0.012,0.00294637820753)
(0.008,0.00130847185346)
(0.004,0.000404779532331)
(0.002,0.000108990760118)
(0.0008,2.44863080683e-05)
(0.0004,9.17733009426e-06)
};
\addplot+[only marks,mark=x,mark options={color=black,fill=black},mark size=3] coordinates {
(0.024,0.0167438058213)
(0.02,0.0104190432586)
(0.016,0.00601877181581)
(0.012,0.00261425013254)
(0.008,0.000954682068198)
(0.004,0.000358527487236)
(0.002,8.52810295768e-05)
(0.0008,1.37638030466e-05)
(0.0004,3.44283445486e-06)
};
\addplot+[only marks,mark=-,mark options={color=black,fill=black},mark size=3] coordinates {
(0.024,0.0164840994848)
(0.02,0.0101336860129)
(0.016,0.00570937416964)
(0.012,0.00228212205756)
(0.008,0.000600892282937)
(0.004,0.000312275442142)
(0.002,6.15712990355e-05)
(0.0008,3.04129802497e-06)
(0.0004,0)
};


	\end{axis}
\end{tikzpicture}
\end{minipage}
\begin{minipage}[c]{0.01\textwidth}
$\,$
\end{minipage}
\begin{minipage}[r]{0.3\textwidth}
\begin{tikzpicture}[scale=.6]
    \begin{axis}[ylabel=$\Delta$, xlabel=$\lambda_2$, xmode=log, ymode=log, ymin=0.0000001, ymax=1, xmin=.0001, xmax=.1, title={$\omega_3$}]
        \addplot+[only marks,mark=-,mark options={color=black,fill=black},mark size=3] coordinates {
(0.024,0.146958332314)
(0.02,0.0871541946763)
(0.016,0.0436337926949)
(0.012,0.0285122347656)
(0.008,0.00683889418042)
(0.004,0.00148053388164)
(0.002,0.000392912394702)
(0.0008,9.61021378078e-05)
(0.0004,2.8406017238e-05)
};
\addplot+[only marks,mark=x,mark options={color=black,fill=black},mark size=3] coordinates {
(0.024,0.146621666397)
(0.02,0.0868065323386)
(0.016,0.043275739207)
(0.012,0.0281443498883)
(0.008,0.00646167373177)
(0.004,0.00141339252209)
(0.002,0.000357186769507)
(0.0008,7.73818127566e-05)
(0.0004,1.58880076384e-05)
};
\addplot+[only marks,mark=-,mark options={color=black,fill=black},mark size=3] coordinates {
(0.024,0.146285000479)
(0.02,0.0864588700009)
(0.016,0.042917685719)
(0.012,0.0277764650109)
(0.008,0.00608445328312)
(0.004,0.00134625116254)
(0.002,0.000321461144311)
(0.0008,5.86614877054e-05)
(0.0004,3.36999803885e-06)
};


	\end{axis}
\end{tikzpicture}
\end{minipage}
\begin{center}
\footnotesize
\begin{tabular}{cccccccccc}
$\gamma$ & 0.4 & 0.5 & 0.6 & 0.7 & 0.8 & 0.9 & 0.95 & 0.98 & 0.99 \\ \hline\hline
$\lambda_1$ & 0.184 & 0.220 & 0.256 & 0.292 & 0.328 & 0.364 & 0.382 & 0.3928 & 0.3964 \\
$\lambda_2$ & 0.024 & 0.020 & 0.016 & 0.012 & 0.008 & 0.004 & 0.002 & 0.0008 & 0.0004
\end{tabular}
\normalsize
\end{center}
\caption{
Absolute difference  $\Delta$ (marked $\times$) between PDE and MC results with $3\sigma$ error bounds ($\scriptstyle -$), for varying $\lambda_2$ and payoff weights $\omega_1$, $\omega_2$ and $\omega_3$ from (\ref{omega1}), (\ref{omega2}) and (\ref{omega3}), using the first order expansion method. The best fit exponents for the differences are $2.04\pm 0.11$, $2.03\pm 0.11$ and $2.21\pm 0.17$ ($95\%$ confidence bounds from a linear regression on a log-log scale).}
\label{fig:varlambda2}
\end{figure}
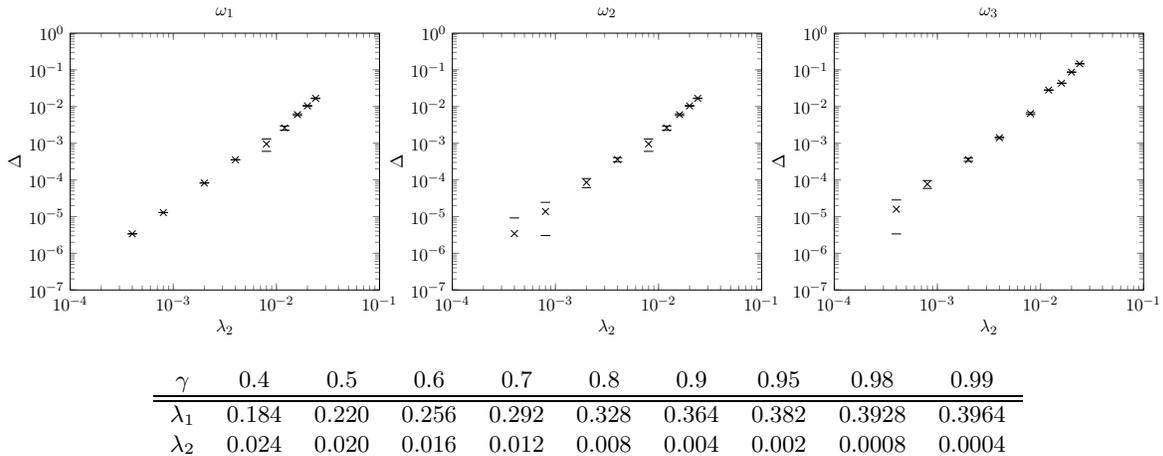
    
    We consider first a European arithmetic basket option with maturity $T = 1$ and payoff
    \begin{eqnarray*}
        g(S) &=& \max\left( \sum_{i=1}^N \omega_i S_i - K , 0 \right),
    \end{eqnarray*}
    where $\omega \in \mathbb{R}^N$ is a vector of weights and $K$ is the strike, which we set to $K=100$. 
We study the solution at the coordinates $S_{i,0} = 100$ for all $i$.
    The numerical tests are for $N = 10$ and $\sigma_i = 0.2$ for $i=1,\ldots, N$. Other details are as specified in Section \ref{subsec:defmodel}.
   
    Because of the structure of $\Sigma$ with one dominant eigenvalue, we choose $r=1$ and expand around the point $\lambda^0 = (\lambda_1,0,\ldots,0)$. Figure \ref{fig:varlambda2} shows results for 
    \begin{eqnarray}
\label{omega1}
        \omega_1 &=& (1/10,1/10,1/10,1/10,1/10,1/10,1/10,1/10,1/10,1/10), \\ \label{omega2}
        \omega_2 &=& (4/30,4/30,4/30,4/30,4/30,2/30,2/30,2/30,2/30,2/30), \\ \label{omega3}
        \omega_3 &=& (1/4,1/4,1/4,1/4,1/4,1/4,1/4,-1/4,-1/4,-1/4).
    \end{eqnarray}
Plotted is the difference $\Delta$ between the first order expansion solution $u_{1,1}$ from Example \ref{ex-first-order}, and a Monte Carlo estimate of the true solution, such that $\Delta$ can be considered an estimate of the error $\widehat{u}_{1,1}$ from (\ref{eqn:truncerr}), with $r=m=1$,
and Theorem \ref{thm:BoundPDEDiff2}.
By varying $\gamma$ in (\ref{eqn:corr}), we vary $\lambda_2=\lambda_3=\ldots=\lambda_N$, and can therefore examine the expansion error as a function of the eigenvalues.

The one- and two-dimensional PDEs for $u^{\{1\}}$ and $u^{\{1,k\}}$ for $k=2,\ldots,N$ were solved numerically as described in Section
\ref{subsec:numericalPDEs}. Specifically, we use the Peaceman-Ratchford ADI scheme \cite{MR55}
with $J = 800$ grid points in each spatial direction and $M=50$ time steps.
The MC estimator was constructed by exact sampling of the joint lognormal distribution of the asset price vector $S_T$ from (\ref{blackscholes}) and computation of the sample mean of the payoff $g(S_T)$ over $10^{10}$ samples to find $V(S_0,0)$ by (\ref{eq:FairPrice}). 
	
		For very small $\lambda_2$ as we need for the study of the asymptotic error, the variance of straightforward MC was too large to assess the accuracy of the expansion approximation.
		To verify that $u-u_{1,1} = O(\lambda_2^2)$ by a standard MC estimator for $u$ and PDE solution for $u_{1,1}$, one needs a standard error of $O(\lambda_2^2)$ or smaller for $u$, i.e.\ $O(\lambda_2^4)$ samples. For the four smallest values of $\lambda_2$ in the tests, $\lambda_2\le 0.004$ and $\lambda_2^4 \approx 2.5 \times 10^{-10}$, such that the results for a feasible number of samples were too noisy for a meaningful comparison.
		We thus chose a control variate approach as outlined in Section \ref{subsec:varred}, in which the difference between the full and approximated solution was calculated via a MC estimator using the same paths for both terms.

Figure \ref{fig:varlambda2} suggests that the difference $\Delta$ between the truncated and full solution follows a power law of the form
        $\Delta \sim  \lambda_2^p$. 
The results reported in Figure \ref{fig:varlambda2} are consistent with the theoretical order of $2$ established in Theorem \ref{thm:BoundPDEDiff2}. In absolute terms, the error for practically relevant correlations around 0.5 is approximately $10^{-2}$ (with the exception of the more challenging payoff given by $\omega^3$), so approximately one basis point in relation to $S_{i,0}=100$.

    \subsection{Convergence of the second order expansion}
    \label{subsec:numsecond}
    
    We consider the same arithmetic basket option and model as in Section \ref{subsec:numfirst}, but with $N=5$ (as opposed to $N=10$ in the first-order case). This allows us to examine the expansion order with significantly reduced computation time, as more (now $O(N^2)$ instead of $O(N)$) and higher-dimensional (now 3-dimensional) PDEs need to be solved in the second-order expansion $u_{1,2}$ from Example \ref{pro:SecondOrderExp}.
    
    The one-, two-, and three-dimensional PDEs for, respectively, $u^{\{1\}}$, $u^{\{1,k\}}$, and $u^{\{1,k,l\}}$ for $2 \le k<l\le N$ where solved numerically as described in Section \ref{subsec:numericalPDEs}. In addition to the one-and two-dimensional PDEs already discussed in Section
 \ref{subsec:numfirst}, we used Brian's ADI method \cite{B61,CCC91} with $J = 500$ grid points in each spatial direction and $M=50$ time steps for the numerical solution of the three-dimensional PDE.
    The MC estimator was constructed in exactly the same way as in Section \ref{subsec:numfirst}.
     For the four smallest values of $\lambda_2$ we again directly compute the difference between full and approximate solution via MC, as described in Section \ref{subsec:varred}.

     \begin{figure}[tbp]
\begin{center}
    \begin{minipage}[c]{0.45\textwidth}
    \begin{center}
\begin{tikzpicture}[scale=.8]
    \begin{axis}[ylabel=$\Delta$, xlabel=$\lambda_2$, xmode=log, ymode=log, ymax=10, ymin=0.00000001, xmin=.0001, xmax=.1, title={$\omega_1$}]
       \addplot+[only marks,mark=-,mark options={color=black,fill=black},mark size=3] coordinates {
(0.024,0.988523649401)
(0.02,0.507907594488)
(0.016,0.130546885728)
(0.012,0.171003917159)
(0.008,0.0226664721942)
(0.004,0.0033110560836)
(0.002,0.00113434398665)
(0.0008,8.92034109225e-05)
(0.0004,8.51814269697e-06)
};
\addplot+[only marks,mark=x,mark options={color=black,fill=black},mark size=3] coordinates {
(0.024,0.987832217757)
(0.02,0.507250115531)
(0.016,0.129927069209)
(0.012,0.170426514882)
(0.008,0.0221380013347)
(0.004,0.00284117246681)
(0.002,0.000699257000001)
(0.0008,4.58750109225e-05)
(0.0004,1.86204934671e-06)
};
\addplot+[only marks,mark=-,mark options={color=black,fill=black},mark size=3] coordinates {
(0.024,0.987140786113)
(0.02,0.506592636574)
(0.016,0.12930725269)
(0.012,0.169849112606)
(0.008,0.0216095304752)
(0.004,0.00237128885001)
(0.002,0.000264170013348)
(0.0008,2.54661092246e-06)
(0.0004,0)
};


	\end{axis}
\end{tikzpicture}
\end{center}
\end{minipage}
\begin{minipage}[c]{0.45\textwidth}
    \begin{center}
\begin{tikzpicture}[scale=.8]
    \begin{axis}[ylabel=$\Delta$, xlabel=$\lambda_2$, xmode=log, ymode=log, ymax=10, ymin=0.00000001, xmin=.0001, xmax=.1, title={$\omega_2$}]
      \addplot+[only marks,mark=-,mark options={color=black,fill=black},mark size=3] coordinates {
(0.024,1.15165485711)
(0.02,0.448088969728)
(0.016,0.0639110203066)
(0.012,0.184477118479)
(0.008,0.0460831684078)
(0.004,0.0086788078288)
(0.002,0.00200245496066)
(0.0008,0.000112527789359)
(0.0004,5.82511550017e-06)
};
\addplot+[only marks,mark=x,mark options={color=black,fill=black},mark size=3] coordinates {
(0.024,1.15089181401)
(0.02,0.447369003367)
(0.016,0.0632383399607)
(0.012,0.18385723706)
(0.008,0.0455236655028)
(0.004,0.00819120937166)
(0.002,0.00195197577498)
(0.0008,0.000105033036979)
(0.0004,3.71525744747e-06)
};
\addplot+[only marks,mark=-,mark options={color=black,fill=black},mark size=3] coordinates {
(0.024,1.15012877091)
(0.02,0.446649037006)
(0.016,0.0625656596147)
(0.012,0.183237355642)
(0.008,0.0449641625977)
(0.004,0.00770361091451)
(0.002,0.00190149658931)
(0.0008,9.75382845988e-05)
(0.0004,1.60539939477e-06)
};

	\end{axis}
\end{tikzpicture}
\end{center}
\end{minipage}
\footnotesize
\begin{tabular}{cccccccccc}
$\gamma$ & 0.4 & 0.5 & 0.6 & 0.7 & 0.8 & 0.9 & 0.95 & 0.98 & 0.99 \\ \hline\hline
$\lambda_1$ & 0.104 & 0.120 & 0.136 & 0.152 & 0.168 & 0.184 & 0.192 & 0.1968 & 0.1984\\
$\lambda_2$ & 0.024 & 0.020 & 0.016 & 0.012 & 0.008 & 0.004 & 0.002 & 0.0008 & 0.0004
\end{tabular}
\normalsize
\end{center}
\caption{Absolute 
difference $\Delta$ (marked $\times$) between PDE and MC results with $3\sigma$ error bounds ($\scriptstyle -$) for varying $\lambda_2$ and payoff weights $\omega_1$ and $\omega_2$ from (\ref{omega1}) and (\ref{omega2}), using the second order expansion method. The best fit exponents for the absolute differences are $3.04\pm 0.28$ and $2.76\pm 0.45$.
}
\label{fig:varlambda2_2}
\end{figure}
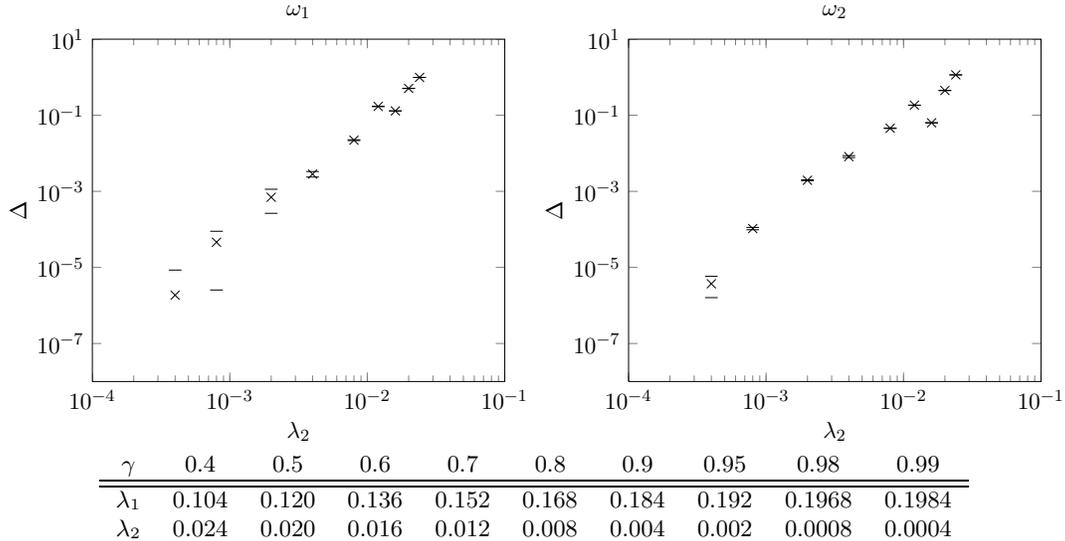
    
The matrix of eigenvectors is
\begin{eqnarray*}
Q &=& \left(\begin{array}{ccccc}
0.4472&  \star &  \star &\star &\star\\
0.4472&  \star &  \star &\star &\star\\
0.4472&  \star &  \star &\star &\star\\
0.4472&  \star &  \star &\star &\star\\
0.4472&  \star &  \star &\star &\star 
\end{array}\right),
\end{eqnarray*}
 where the $\star$ entries can be any set of normalized, orthogonal vectors which span the $N-1$ dimensional subspace orthogonal to the first eigenvector (since $\lambda_2 = \ldots = \lambda_N$). For weight vectors $\omega$ closely aligned with the first eigenvector, such as $\omega = (1/5,1/5,1/5,1/5,1/5)$, the second order scheme is extremely accurate and the resulting absolute error is considerably less than $10^{-5}$ even for the larger of the values of $\lambda_2$ considered here. We instead choose payoff weights
    \begin{eqnarray*}
        \omega_1 &=& (1,-1,1,-1,1), \\
        \omega_2 &=& (3/2,3/2,-1/2,-1/2,-1),
    \end{eqnarray*}
   for which the error is relatively large, since they are not aligned with any of the eigenvectors.\footnote{The angle between $q_1$ and $\omega_1$ and that between $q_1$ and $\omega_2$ are both just under $80\deg$, i.e., closer to orthogonal than parallel.}
This allows us to show convergence across a wide range of $\gamma$ values (see (\ref{eqn:corr})) in Figure \ref{fig:varlambda2_2}. The data points again appear to follow a power law $\Delta \sim \lambda_2^p$. 
The best fit exponents $p$
from Figure \ref{fig:varlambda2_2}
compare well to the theoretical order of $3$.
    
%
%
%
%
%

\subsection{Computational times}

In this section we briefly discuss the computational times for the expansion solutions in relation to Monte Carlo estimation.
Table \ref{tab:runtimes} gives CPU times for the set-up in Sections \ref{subsec:numfirst} and \ref{subsec:numsecond}.

\begin{table}
\begin{center}
\small
\begin{tabular}{c||c|c||c|c}
Method & MC ($N=10$) & ADI ($N=10$, $s=1$) & MC ($N=5$) & ADI ($N=5$, $s=2$) \\ \hline \hline
Paths / Points & $10^{10}$ & $M=50$, $J=800$ & $10^{10}$ & $M=50$, $J=500$ \\
CPU time (sec) & 17902 & 68 & 7737 & 36445 
\end{tabular}
\end{center}
\caption{Computational times on a  2.8 GHz Intel Core 2 Duo processor with 8 GB 1067 MHz DDR3, running Matlab R2016b.
The ADI schemes use $M$ time steps and $J$ mesh intervals in each direction, i.e.\ $6.4\times 10^5$ unknowns in the 2-dimensional case ($s=1$)
and $1.25\times 10^8$ unknowns in the three-dimensional case ($s=2$).}
\label{tab:runtimes}
\end{table}

For the first order expansion ($r=1$, $s=1$) of Section \ref{subsec:numfirst}, CPU times are dominated by the numerical solution of the nine
two-dimensional PDEs involved, whereas for the second order expansion ($r=1$, $s=2$) of Section \ref{subsec:numsecond},
the dominant contribution comes from the six three-dimensional PDEs.

The time for the first order expansion with ten assets is roughly one minute, compared to five hours for Monte Carlo,
and for the second order expansion with five assets, roughly ten hours, compared to two hours with Monte Carlo.
A few comments are in order.

The ADI computing times are broadly similar to the ones reported in \cite{in2010adi}
for two-dimensional PDEs and in \cite{haentjens2010adi} for three-dimensional PDEs, extrapolating to our finer meshes and taking into account minor differences in the set-up.

The numbers of Monte Carlo paths, ADI time steps and mesh points were chosen to make the numerical error much smaller than the expansion error, as the latter is the one of interest for the purposes of this paper. In practical applications, there is no benefit in reducing the Monte Carlo and finite difference error far below the expansion error, so that a significantly smaller number of paths and points can be used, reducing the computing time.

Specifically, the finite difference error for the above setup is substantially below $10^{-3}$. For a practically acceptable accuracy of one basis-point, i.e.\ $1/100 \ \%$ of $S_0=100$, equalling $10^{-2}$, it would suffice to take a quarter of the steps, $J=200$ and $M=12$, reducing the computational time for the first order expansion by a factor of $2^6$ to about 1 sec. Equally, a $95\%$ Monte Carlo confidence interval with width $2 \cdot 10^{-2}$ could already be obtained with $10^7$ samples, in 18 secs.  

The Monte Carlo estimator samples the joint log-normal distribution in a single time step, which would not be possible if a more general model was chosen, which biases the computing times by a factor of 12 or 50, respectively, in favour of the Monte Carlo method. So in a realistic set-up, given the estimated run-times from the previous paragraph, the combined first order expansion finite difference method will be about a factor of 
$12\times 18 \approx 200$ faster than a Monte Carlo estimator with similar accuracy.


\subsection{The effect of kinks}

    To demonstrate how the solution behaves around kink points, we switch to a geometric basket option with $N=10$ and payoff
    \begin{eqnarray}
    \nonumber
        g(S) &=& \max\left( \prod_{i=1}^N S_i^{\omega_i} - K , 0 \right) = \max\left( \exp\left(\sum_{i=1}^N \omega_i\log{S_i}\right) - K , 0 \right).
    \end{eqnarray}
    See Example  \ref{ex:GeometricBasket} for the two-dimensional case.
   Due to the log-normality, this results in a linear kink, for which we can easily choose vectors of weights $\omega$ which are either exactly parallel, orthogonal, or otherwise oriented in relation to the eigenvectors of $\Sigma$ to best examine this effect.
   
    For $\omega$ which are not orthogonal to $q_1$, such as $\omega = q_1$ or $\omega = q_1 + q_2$, this expansion scheme is extremely accurate, even at the kink point $K = 1$. This was expected from the previous section. We thus focus on the orthogonal case where no smoothing occurs
(see Example \ref{ex:GeometricBasket} in Appendix \ref{subsec:exnonsmooth}),
    namely,
\footnote{These two seemingly arbitrary eigenvectors
are taken from a MATLAB decomposition for $\Sigma$.}
{\small
    \begin{eqnarray*}
        \omega_1 &=& \left(
 -0.1160,
0.0929,
-0.6527,
-0.1121,
0.6986,
0.2091,
-0.0438,
-0.0758,
0.0000,
0.000
\right), \\
\omega_2 &=& \left(
 0.1130,
-0.0607,
-0.1708,
-0.2057,
 0.8971,
-0.2467,
-0.1831,
-0.1085,
-0.0345,
-0.0001
\right).
    \end{eqnarray*}

    }
   The kink point is now located at strike price $K = 1$. Figure \ref{fig:kinkpoint} shows that there is indeed a marked increase in the error size around the kink and that the scheme is very accurate away from it.
    The estimated convergence order away from the kink (see Figure \ref{fig:kinkpoint}) is consistent with the theoretically predicted order of 2.
At the kink, the order of convergence appears reduced to below 1, although we expect it to be 1 from Example  \ref{ex:GeometricBasket}.

     \begin{figure}[tbp]
\begin{center}
\hspace{-0.04 \textwidth}
    \begin{minipage}[c]{0.3\textwidth}
    \begin{center}
\begin{tikzpicture}[scale=.6]
    \begin{axis}[ylabel=$\Delta$, xlabel=$K$, ymode=log]
\addplot[mark=x,mark options={color=black,fill=black},mark size=3,blue] coordinates {
(0.95,0.00347543763529)
(0.98,0.0227607645389)
(0.99,0.0445079901946)
(1.0,0.0779470151649)
(1.01,0.0445012232155)
(1.02,0.0224818562153)
(1.05,0.00519052849084)
};
\addplot[mark=x,mark options={color=black,fill=black},mark size=3,red] coordinates {
(0.95,0.000294286922009)
(0.98,0.0142556033338)
(0.99,0.0237388171223)
(1.0,0.0458659348561)
(1.01,0.023746534674)
(1.02,0.0140965315816)
(1.05,0.000374546773123)
};
	\end{axis}
\end{tikzpicture}
\end{center}
\end{minipage}
\hspace{0.02 \textwidth}
    \begin{minipage}[c]{0.3\textwidth}
    \begin{center}
\begin{tikzpicture}[scale=.6]
    \begin{axis}[ylabel=$\Delta$, xlabel=$\lambda_2$, xmode=log, ymode=log]
\addplot+[only marks,mark=-,mark options={color=black,fill=black},mark size=3] coordinates {
(0.024,5.35648340767e-05)
(0.02,3.69859623715e-05)
(0.016,2.15364206363e-05)
(0.012,1.29431474174e-05)
(0.008,9.59614274121e-06)
(0.004,3.37370680599e-06)
(0.004,0.0000017706331392068)
(0.002,0.0000003267212922781)
(0.0008,0.0000000786172656000)
(0.0004,0.0000000353831623376)
};
\addplot+[only marks,mark=x,mark options={color=black,fill=black},mark size=3] coordinates {
(0.024,4.8832803706e-05)
(0.02,3.2679167679e-05)
(0.016,1.7695808042e-05)
(0.012,9.62703431695e-06)
(0.008,6.89669990206e-06)
(0.004,0.0000015384391990540)
(0.002,0.0000002404270619839)
(0.0008,0.0000000393648287369)
(0.0004,0.0000000152060816369)
};
\addplot+[only marks,mark=-,mark options={color=black,fill=black},mark size=3] coordinates {
(0.024,4.41007733354e-05)
(0.02,2.83723729866e-05)
(0.016,1.38551954477e-05)
(0.012,6.31092121653e-06)
(0.008,4.19725706291e-06)
(0.004,0.0000013062452589011)
(0.002,0.0000001541328316896)
};


\end{axis}
\end{tikzpicture}
\end{center}
\end{minipage}
\hspace{0.02 \textwidth}
\begin{minipage}[c]{0.3\textwidth}
    \begin{center}
\begin{tikzpicture}[scale=.6]
    \begin{axis}[ylabel=$\Delta$, xlabel=$\lambda_2$, xmode=log, ymode=log]
\addplot+[only marks,mark=-,mark options={color=black,fill=black},mark size=3] coordinates {
(0.024,0.0768045909375)
(0.02,0.0617744988907)
(0.016,0.0779628792875)
(0.012,0.0617602723576)
(0.008,0.0298314133824)
(0.004,0.0256853454518)
(0.002,0.018890466708701)
(0.0008,0.008617451404051)
(0.0004,0.007169182148139)
};
\addplot+[only marks,mark=x,mark options={color=black,fill=black},mark size=3] coordinates {
(0.024,0.0767732919888)
(0.02,0.061746296976)
(0.016,0.077938013533)
(0.012,0.0617390810957)
(0.008,0.0298144332607)
(0.004,0.0256736231303)
(0.002,0.018890380414471)
(0.0008,0.008617412151614)
(0.0004,0.007169161971058)
};
\addplot+[only marks,mark=-,mark options={color=black,fill=black},mark size=3] coordinates {
(0.024,0.0767419930402)
(0.02,0.0617180950614)
(0.016,0.0779131477786)
(0.012,0.0617178898338)
(0.008,0.0297974531389)
(0.004,0.0256619008088)
(0.002,0.018890294120241)
(0.0008,0.008617372899177)
(0.0004,0.007169141793977)
};


\end{axis}
\end{tikzpicture}
\end{center}
\end{minipage}
\end{center}
\caption{Left: Absolute error  $\Delta$ (marked $\times$) between PDE and MC solution for a geometric basket with payoff weight $\omega_1$ (bottom) and $ \omega_2$ (top) around the kink point at $K=1$ for $\gamma = 0.6$ for fixed spot price and varying strike. The option value is around $5\cdot 10^{-2}$. Centre/Right: Absolute  
errors $\Delta$ (marked $\times$) with $3\sigma$ error bounds ($\scriptstyle -$)
for $\omega_1$ with different values for $\lambda_2$ at $K = 0.5$ (centre) and $K=1$ (right). The best fit exponents for the errors are $2.01\pm 0.13$ and $0.61\pm 0.10  $ ($95\%$ confidence bounds from a linear regression on a log-log scale).
}
\label{fig:kinkpoint}
\end{figure}
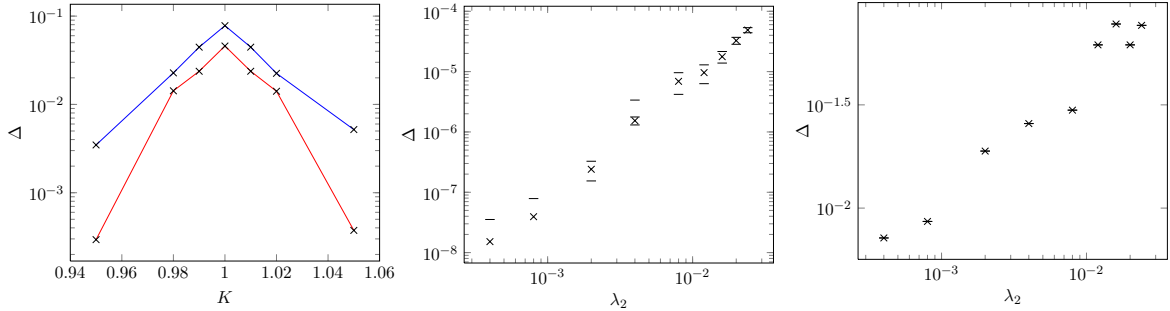

It is worth emphasising that a high order of convergence is achieved in all but the at-the-money, orthogonal case, despite the non-smooth initial data. This is in line with, and empirically validates, our theoretical analysis in the previous sections.

  \section{Conclusion and outlook}\label{sec:Conclusion}

	In this paper, we provide a thorough theoretical analysis of the PCA-based expansion approach for 
	parabolic PDEs first introduced in \cite{RW07}. Building on our previous work in \cite{RW13}, we give a description of a method which is general enough to cover a wide range of important cases and yet allows concrete error bounds in special cases, such as the ones given in Section \ref{sec:ErrorBounds}. We illustrate our theoretical results with numerical experiments, which show accuracy and convergence in line with theoretical predictions.

    This work focuses on 
    the case of constant coefficients, where the PDE can be transformed to the $N$-dimensional heat equation
    via rotation into the eigensystem of the covariance matrix. Using 
    PDEs for the expansion error itself,
    a careful analysis of the applicability of the method and the size of the expansion error was possible.

    For non-constant coefficients, one can approximate the PDE via localization. 
%
%
	    Consider an $N$-dimensional linear parabolic PDE
    \begin{eqnarray}\label{eq:PDEDef1}
        \frac{\partial v}{\partial t} &=& \sum_{i,j=1}^N \Sigma_{ij}(x,t) \frac{\partial^2 v}{\partial x_i \partial x_j} + \sum_{i=1}^N \mu_i(x,t) \frac{\partial v}{\partial x_i}, \label{eq:PDEDef1_2}
    \end{eqnarray}
    where $\Sigma_{ij} = \sigma_i \sigma_j \rho_{ij}$ is a function of $x$ and $t$.
For simplicity, we have omitted a zero order term, but 
a term of the form $c(t) v$ can directly be accounted for by the definition of a new function $\exp(\int_0^t c(\tau) d\tau)v$.

A transformation to the heat equation like in the constant coefficient case of Section \ref{subsec:defmodel} is generally no longer possible.

A simple approximation is to replace $\Sigma(x,t)$ and $\mu(x,t)$ by $\Sigma(x_0,0)$ and $\mu(x_0,0)$, respectively, for some fixed $x_0$. If one is interested only in the solution at $(x_0,t)$ for not too large $t$, this may be justified. One can then proceed as in Section \ref{subsec:defmodel}.
     This approach leads to a localization error 
     which depends on the rate of change of the PDE coefficients in the vicinity of $(x_0,0)$. 
     
     A concrete example where this approach can be successful is the LIBOR Market Model, where the LIBOR rates process has a strongly non-linear drift in the pricing measure, resulting in non-constant coefficients for the first derivatives in the pricing PDE. In \cite{RW13}, the drift is ``frozen'' at the initial value of the process, leading to a constant coefficient PDE, which is then approximated by an expansion. The numerical tests for Bermudan and path-dependent derivatives illustrate the error from the constant coefficient approximation, as well as the expansion error, and indicate that for moderate maturities both are within an acceptable range. 

   Alternatively, one can adapt the PDE expansion method to directly incorporate non-constant PDE coefficients.
By choosing $Q$ as the eigenvectors of $\Sigma(x_0,0)$ and transforming the PDE (\ref{eq:PDEDef1}) to the coordinate system
$z = z(x,t) = Q^T (x + \int_0^t \mu(x_0,s) \, ds)$, one obtains a PDE
    \begin{eqnarray}\label{eq:PDEDef11}
        \frac{\partial u}{\partial t} &=& \frac{1}{2} \sum_{i,j=1}^N \Lambda_{ij}(z,t) \frac{\partial^2 u}{\partial z_i \partial z_j} + \sum_{i=1}^N \kappa_i(z,t) \frac{\partial u}{\partial z_i}   ,     
    \end{eqnarray}
with coefficients $\Lambda$ and $\kappa$ such that $\Lambda_{ij}(z_0,0) = 0$ for $i\neq j$ and $\kappa_i(z_0,0)=0$, where $z_0 = z(x_0,0)$.

The dimensionality of the PDE can be reduced by setting the diffusion and drift coefficients to zero for indices $i \notin \nu$, for
a given index set $\nu$.
More specifically, we set $\tilde{\Lambda}_{ij}=\tilde{\Lambda}_{ji}=0$ and $\tilde{\kappa}_i = 0$ for $i \notin \nu$ and for all $1\le j\le N$,
and $\tilde{\Lambda}_{ij} = {\Lambda}_{ij}$ and $\tilde{\kappa}_i = \kappa_i$ otherwise.
One can then define an approximation $\tilde{u}^\nu(z,t)$ 
as solution to (\ref{eq:PDEDef11}) with $\Lambda$ and $\kappa$ replaced by $\tilde{\Lambda}$ and $\tilde{\kappa}$.

The dimensionality of the PDE is then $|\nu|$. The index sets of interest are of the form $\nu_r = \{1,\ldots, r\}$, $\nu_r\cup \{k\}$ for some $r<k\le N$, or $\nu_r\cup \{k,l\}$ for some $r<k<l\le N$, in the spirit of this paper.

It is shown in \cite{RW17} how to build dimensionwise decompositions in this case and numerical results are presented for variable (in time and space) volatilities and correlations. 
Those preliminary tests in \cite{RW17} show that good accuracy is often already obtained for the first order version even in the variable coefficient setting, but for fast varying coefficients the error can be significantly larger than for constant coefficients.
A sketch of a possible analysis in this case is given  in \cite{RW17}, but theoretical error bounds for variable coefficients remain an open question.

It would be interesting to conduct further tests on practically used local or stochastic volatility models. For options on a baskets of $N$ assets, where each asset price process has its own stochastic volatility, the above method applied directly to the $2N$-dimensional pricing PDE can be expected to work well when the asset prices are strongly correlated, and strongly negatively correlated with their stochastic volatilities, combined with slowly varying volatility. This regime should be close to the constant coefficient case and the eigenvalues of the ``frozen'' covariance matrix will decay rapidly.
Interesting adaptations of the method could take advantage of the fact that the payoff is only a function of the assets and not the volatility, that the inclusion of each stochastic  volatility can be seen as adding a correction, and one might even combine the eigenvalue expansion with asymptotic expansions with respect to fast and slow volatility scales, see \cite{fouque2011multiscale}.

It is conceivable that the method might also be successfully applied to partial integro-differential equations arising from multivariate jump-diffusion processes, if a suitable expansion around a dominant component can be defined, e.g., if the jump process can be interpreted as a time-changed Brownian motion.

\bibliographystyle{plain}
\bibliography{Bibliography}

	
	\appendix
	\section{Proofs of lemmas}

\subsection{Proof of Lemma \ref{lem:PDEerror}}
\label{app:PDEerror}
	\small
	
		By definition,
	    \begin{eqnarray}
	        \frac{\partial}{\partial t}\hat{u}^\xi 
			&=& \sum_{(w,\nu)\in\xi}w\mathcal{L}^\nu u^\nu - \mathcal{L}u\nonumber\\
			&=& \left[\mathcal{L}^{\nu_\xi}-\mathcal{L}^{\nu_\xi}\right]\sum_{(w,\nu)\in\xi}wu^\nu + \sum_{(w,\nu)\in\xi}w\mathcal{L}^\nu u^\nu 
			+ \left[\mathcal{L}^{\nu_\xi}-\mathcal{L}^{\nu_\xi}\right]u - \mathcal{L}u  \nonumber\\
			&=& \mathcal{L}^{\nu_\xi}\left[u^\xi - u\right] + \sum_{(w,\nu)\in\xi}w\left[\mathcal{L}^\nu-\mathcal{L}^{\nu_\xi}\right] u^\nu 
			+ \left[\mathcal{L}^{\nu_\xi}-\mathcal{L}\right]u. 
	    \end{eqnarray}

\subsection{Proof of Lemma \ref{lem:PDEerror1}}
\label{app:PDEerror1}
		    We use Lemma \ref{lem:PDEerror}. The relevant differential operators 
		    are
    \begin{eqnarray*}
    \mathcal{L} = \mathcal{L}^{\{1,\ldots,N\}} = \sum_{k=1}^N{ \lambda_k\frac{\partial^2}{\partial z_k^2}}, \quad
    \mathcal{L}^{\{1,\ldots,r\}} = \sum_{k=1}^r{ \lambda_k\frac{\partial^2}{\partial z_k^2}}, \quad
    \mathcal{L}^{\{1,\ldots,r,k\}} = \mathcal{L}^{\{1,\ldots,r\}} + \lambda_i\frac{\partial^2}{\partial z_k^2},
    \end{eqnarray*}
	and so on. Using this we can rewrite equation (\ref{eq:PDEuhatxi1}) as
    \begin{eqnarray*}
        \frac{\partial}{\partial t} \hat{u}^\xi &=& \mathcal{L}^{\{1,\ldots,r\}}\hat{u}^\xi +
		\sum_{k=r+1}^N{\left[ \mathcal{L}^{\{1,\ldots,r,k\}} - \mathcal{L}^{\{1,\ldots,r\}} \right]u^{\{1,\ldots,r,k\}}} + \left[ \mathcal{L}^{\{1,\ldots,r\}} - \mathcal{L} \right]u \\
		&=& \sum_{k=1}^r{ \lambda_k\frac{\partial^2}{\partial z_k^2}}\hat{u}^\xi + \sum_{k=r+1}^N{ \lambda_k\frac{\partial^2}{\partial z_k^2} \left[ u^{\{1,\ldots,r,k\}}-u \right]}.
    \end{eqnarray*}

\subsection{Proof of Lemma \ref{lem:BoundPDEDiff3}}
\label{app:BoundPDEDiff3}


We start with equation (\ref{eq:PDEuhatxi1})
 \begin{eqnarray*}
	        \frac{\partial}{\partial t} \hat{u}^\xi &=& \mathcal{L}^{\nu_\xi}\hat{u}^\xi + \sum_{(w,\nu)\in\xi}w\left[ \mathcal{L}^{\nu} - \mathcal{L}^{\nu_\xi} \right]u^\nu 
	        + \left[\mathcal{L}^{\nu_\xi}-\mathcal{L}\right]u
	    \end{eqnarray*}
and insert explicitly $\xi$ from (\ref{eq:xiTaylorExpansion2,r}) to obtain
    \begin{eqnarray*}
        \frac{\partial}{\partial t} \hat{u}^\xi &=& \mathcal{L}^{\{1,\ldots,r\}}\hat{u}^\xi + (1+(N-r)(N-r-3)/2)\left[ \mathcal{L}^{\{1,\ldots,r\}} - \mathcal{L}^{\{1,\ldots,r\}} \right]u^{\{1,\ldots,r\}} \nonumber\\
		&& + (2-(N-r))\sum_{k=r+1}^N{\left[ \mathcal{L}^{\{1,\ldots,r,k\}} - \mathcal{L}^{\{1,\ldots,r\}} \right]u^{\{1,\ldots,r,k\}}} \\
		&& + \sum_{k=r+1}^N \sum_{l=k+1}^N {\left[ \mathcal{L}^{\{1,\ldots,r,k,l\}} - \mathcal{L}^{\{1,\ldots,r\}} \right]u^{\{1,\ldots,r,k,l\}}} 
		+ \left[ \mathcal{L}^{\{1,\ldots,r\}} - \mathcal{L}^{\{1,\ldots,N\}} \right]u \\
		&=& \sum_{k=1}^r \lambda_k\frac{\partial^2}{\partial z_k^2}\hat{u}^\xi + 0  - (N-r-2) \sum_{k=r+1}^N \lambda_k\frac{\partial^2}{\partial z_k^2}u^{\{1,\ldots,r,k\}} \\
		&& + \sum_{k=r+1}^N\sum_{l=k+1}^N \left(\lambda_k\frac{\partial^2}{\partial z_k^2}+\lambda_l\frac{\partial^2}{\partial z_l^2}\right)u^{\{1,\ldots,r,k,l\}} - \sum_{k=r+1}^N \lambda_k\frac{\partial^2}{\partial z_k^2} u
    \end{eqnarray*}
Observing that the double sum contains exactly $(N-r)(N-r-1)$ terms we can rewrite this as
    \begin{eqnarray*}
        \frac{\partial}{\partial t} \hat{u}^\xi &=& \sum_{k=1}^r \lambda_k\frac{\partial^2}{\partial z_k^2}\hat{u}^\xi  - (N-r-2)\sum_{k=r+1}^N \lambda_k\frac{\partial^2}{\partial z_k^2}u 
        - (N-r-2) \sum_{k=r+1}^N \lambda_k\frac{\partial^2}{\partial z_k^2}[u^{\{1,\ldots,r,k\}} - u] \\
		&& + \sum_{k=r+1}^N\sum_{l=k+1}^N \left(\lambda_k\frac{\partial^2}{\partial z_k^2}+\lambda_l\frac{\partial^2}{\partial z_l^2}\right)[u^{\{1,\ldots,r,k,l\}}-u] 
		+ (N-r-1)\sum_{k=r+1}^N \lambda_k\frac{\partial^2}{\partial z_k^2}u - \sum_{k=r+1}^N \lambda_k\frac{\partial^2}{\partial z_k^2} u\\
		&=& \sum_{k=1}^r \lambda_k\frac{\partial^2}{\partial z_k^2}\hat{u}^\xi 
		- (N-r-2) \sum_{k=r+1}^N \lambda_k\frac{\partial^2}{\partial z_k^2}[u^{\{1,\ldots,r,k\}} - u] 
		+ \sum_{k=r+1}^N\sum_{l=k+1}^N \left(\lambda_k\frac{\partial^2}{\partial z_k^2}+\lambda_l\frac{\partial^2}{\partial z_l^2}\right)[u^{\{1,\ldots,r,k,l\}}-u].
    \end{eqnarray*}
In other words, $\hat{u}^\xi$ solves a non-homogenous $r$-dimensional heat equation with source term
\begin{eqnarray*}
        f &=& \sum_{k=r+1}^N\sum_{l=k+1}^N \left(\lambda_k\frac{\partial^2}{\partial z_k^2}+\lambda_l\frac{\partial^2}{\partial z_l^2}\right)[u^{\{1,\ldots,r,k,l\}}-u]
        - (N-r-2) \sum_{k=r+1}^N \lambda_k\frac{\partial^2}{\partial z_k^2}[u^{\{1,\ldots,r,k\}} - u].
    \end{eqnarray*}
We know from (\ref{laststep}) in the proof of the first order bound that $\hat{u}^{\{1,\ldots,r,k\}} = u^{\{1,\ldots,r,k\}} - u$ is the solution to
	    \begin{eqnarray}
			\frac{\partial}{\partial t}\hat{u}^{\{1,\ldots,r,k\}} &=& \sum_{i\in\{1,\ldots,r,k\}}\lambda_i\frac{\partial^2}{\partial z_i^2}\hat{u}^{\{1,\ldots,r,k\}} -  \sum_{i=r+1,i\neq k}^N
\lambda_i\frac{\partial^2}{\partial z_i^2}u.
	    \end{eqnarray}
Using the Green's function $\Phi^{\{1,\ldots,r,k\}}$ associated with $\mathcal{L}^{\{1,\ldots,r,k\}}$,
\begin{eqnarray*}
    \hat{u}^{\{1,\ldots,r,k\}}(z,t) &=&  -\int_0^T \int_{\mathbb{R}^N} \Phi^{\{1,\ldots,r,k\}}(y,s) \sum_{i=r+1,i\neq k}^N \lambda_i\frac{\partial^2}{\partial z_i^2} u(z-y,t-s) dy ds.
        \end{eqnarray*}
    Similarly, it follows
    \begin{eqnarray*}
    \hat{u}^{\{1,\ldots,r,k,l\}}(z,t) &=& -\int_0^T \int_{\mathbb{R}^N} \Phi^{\{1,\ldots,r,k,l\}}(y,s) \sum_{i=r+1,i\neq k,l}^N \lambda_i\frac{\partial^2}{\partial z_i^2} u(z-y,t-s) dy ds.
\end{eqnarray*}
This allows us to expand the expression for $f$. We leave out the coordinates to increase readability.
\begin{eqnarray*}
        f &=& -\sum_{k=r+1}^N\sum_{l=k+1}^N \left(\lambda_k\frac{\partial^2}{\partial z_k^2}+\lambda_l\frac{\partial^2}{\partial z_l^2}\right) \int_0^T \int_{\mathbb{R}^N} \Phi^{\{1,\ldots,r,k,l\}} \sum_{i=r+1,i\neq k,l}^N \lambda_i\frac{\partial^2}{\partial z_i^2} u \\
        && + (N-r-2) \sum_{k=r+1}^N \lambda_k\frac{\partial^2}{\partial z_k^2} \int_0^T \int_{\mathbb{R}^N} \Phi^{\{1,\ldots,r,k\}} \sum_{i=r+1,i\neq k}^N \lambda_i\frac{\partial^2}{\partial z_i^2} u \\
        &=& -\sum_{l>k\geq r+1}^N\sum_{i=r+1,i\neq k,l}^N \int_0^T \int_{\mathbb{R}^N}  \lambda_k\lambda_i \Phi^{\{1,\ldots,r,k,l\}} \frac{\partial^4}{\partial z_k^2\partial z_i^2} u + \lambda_l\lambda_i \Phi^{\{1,\ldots,r,k,l\}} \frac{\partial^4}{\partial z_l^2\partial z_i^2} u \\
        && + (N-r-2) \sum_{k=r+1}^N\sum_{i=r+1,i\neq k}^N \int_0^T \int_{\mathbb{R}^N} \lambda_k\lambda_i \Phi^{\{1,\ldots,r,k\}} \frac{\partial^4}{\partial z_k^2\partial z_i^2} u \\
    \end{eqnarray*}
Both lines contain $(N-r)(N-r-1)(N-r-2)$ individual terms. This allows us to combine them into
\begin{eqnarray}
\nonumber
        f &=& -\sum_{l>k\geq r+1}^N\sum_{i=r+1,i\neq k,l}^N \int_0^T \int_{\mathbb{R}^N}  \lambda_k\lambda_i [\Phi^{\{1,\ldots,r,k,l\}} - \Phi^{\{1,\ldots,r,k\}}] \frac{\partial^4}{\partial z_k^2\partial z_i^2} u 
        + \,\lambda_l\lambda_i [\Phi^{\{1,\ldots,r,k,l\}} - \Phi^{\{1,\ldots,r,l\}}] \frac{\partial^4}{\partial z_l^2\partial z_i^2} u.
        \label{lasteqnapp}
    \end{eqnarray}
    Now from (\ref{eq:auxPDE}),
 \begin{eqnarray*}
        \tilde{u}^{\{1,\ldots,r,k\}}(z,t) &=&  \int_0^T \int_{\mathbb{R}^N} \Phi^{\{1,\ldots,r,k\}}(y,s) u(z-y,t-s) dy ds, \\
        \frac{\partial^4}{\partial z_k^2\partial z_i^2} \tilde{u}^{\{1,\ldots,r,k\}}(z,t) &=&  \int_0^T \int_{\mathbb{R}^N} \Phi^{\{1,\ldots,r,k\}}(y,s) 
        \frac{\partial^4}{\partial z_k^2\partial z_i^2} u(z-y,t-s) dy ds,
 \end{eqnarray*}
 and similar for $\tilde{u}^{\{1,\ldots,r,k\}}$.
Inserting this in (\ref{lasteqnapp}), the result follows.    

\section{Analytical examples with non-smooth data}
\label{subsec:exnonsmooth}

For illustration, we study the two-dimensional Black-Scholes model, where the value function of a European option
on a pair of assets $S_1, S_2$, with payoff $h(S_1,S_2)$ at time $T$ is given by
$\exp(-r(T-t)) V(S_1,S_2,t)$, where
    \begin{eqnarray}
    \label{eqn:bs2d}
\hspace{-0.8 cm}         \frac{\partial V}{\partial t} + \left({\frac{1}{2}} \sigma_1^2 S_1^2 \frac{\partial^2 V}{\partial S_1^2} +
         \sigma_{12} S_1 S_2
         \frac{\partial^2 V}{\partial S_1 \partial S_2}
          +  {\frac{1}{2}} \sigma_2^2 S_2^2 \frac{\partial^2 V}{\partial S_2^2} \right)
         + r \left( S_1 \frac{\partial V}{\partial S_1}  + S_2 \frac{\partial V}{\partial S_2}
         \right)
        &\!\!\!=&\!\!\! 0,\!\!\!\!\!\! \\
V(S_1,S_2,,T) &\!\!\!=&\!\!\! h(S_1,S_2), \!\!\!\!\!\!
        \nonumber
    \end{eqnarray}
and $\sigma_{12} = \rho  \sigma_1  \sigma_2$.
		
We exploit the log-normality of the Black-Scholes model by using logarithmic coordinates
$(x_1,x_2) = (\log S_1, \log S_2)$.
We now make a  specific choice of parameters which simplifies the equations below, namely $\sigma_1^2 = \sigma^2_2 = 2 r = \sigma^2$ for some $\sigma>0$.
For a general choice of parameters, a further transformation of (\ref{eqn:bs2d}) to eliminate the drift term is needed, but we avoid this complication here as it does not change the essence of the problem.
Elementary calculation shows that  the rotation
	\begin{eqnarray*}
	    \left( \begin{array}{c} z_1 \\ z_2 \end{array} \right) &=& 
	    \left( \begin{array}{rr} 1 & 1 \\ 1 & -1 \end{array} \right) 
	    \left( \begin{array}{c} x_1 \\ x_2 \end{array} \right)
	\end{eqnarray*}
leads to a particularly simple form, the heat equation
    \begin{eqnarray*}
        \frac{\partial u}{\partial t} &=& \lambda_1 \frac{\partial^2 u}{\partial z_1^2} + \lambda_2 \frac{\partial^2 u}{\partial z_2^2},
         \\
        u(\cdot,0) &=& g
    \end{eqnarray*}
    for some $g$, in backward time $t \rightarrow T-t$.
Here,
\begin{eqnarray*}
\lambda_1 = \sigma^2 (1+\rho), \qquad \lambda_2 = \sigma^2 (1-\rho),
\end{eqnarray*}
so that for $\rho$ close to $1$ or $-1$, i.e., near perfect correlation or anti-correlation, the problem becomes close to one-dimensional and the asymptotic expansion can be expected to work particularly well.

The purpose of the following examples is to show that the optimal convergence order is only lost in degenerate cases, which we will describe precisely.
Because of the symmetry, we only expand in $\lambda_2$ without loss of generality.

    \begin{ex}[Digital call on geometric basket] \label{ex:DigitalGeometricBasket}
    
	Consider 
	the payoff function
	\begin{eqnarray*}
	    h(s_1,s_2) &=& \mathds{1}_{[0,\infty)}(s_1^{\mu_1}s_2^{\mu_2}-2)
	\end{eqnarray*}
	with real parameters $\mu_1$ and $\mu_2$, which leads to
	\begin{eqnarray*}
	    g(z_1,z_2) &=& \mathds{1}_{[0,\infty)}\left(\exp\left[\frac{\mu_1+\mu_2}{2}z_1 + \frac{\mu_1-\mu_2}{2}z_2\right] - 2\right).
	\end{eqnarray*}
	
	Note that the level curves of the argument of $\mathds{1}$ are straight lines. We consider this as a ``local'' (i.e., linearised) model for the more general case where discontinuities occur along a smooth curve.
	 
	We will consider three different choices of $\mu_1$ and $\mu_2$, and  analyse the expansion in $\lambda_2$, for fixed $\lambda_1$.
	
	$\,$
	
	\noindent
{\bf	Case 1:} $\mu_1 = 1$, $\mu_2 = 1$ (i.e., gradient of payoff orthogonal to expansion direction).
	
	$\,$
	
	\noindent
	In this case, the value of the boundary function depends only on $z_1$ and correspondingly
	
	\begin{eqnarray*}
		\frac{\partial}{\partial \lambda_2} u(z,t,\lambda)
	    &=& 0.
	\end{eqnarray*}
	
	$\,$
	
	\noindent
{\bf	Case 2:} $\mu_1 = 1$, $\mu_2 = -1$ (i.e., gradient of payoff parallel to expansion direction).
		
	$\,$
	
	\noindent
	The value of the boundary function depends only on $z_2$. We find that the derivative in $\lambda_2$ exists and, by a lengthy calculation, is given by
	
	\begin{eqnarray*}
		\frac{\partial}{\partial \lambda_2} u(z,t,\lambda) 
	    &=& -\frac{z_2-\log{2}}{2\sqrt{2t}\sqrt{\lambda_2}^3} \cdot \frac{\exp(-(z_2-\log{2})^2/4t\lambda_2)}{\sqrt{2\pi}},
	\end{eqnarray*}
	which goes to $0$ for $\lambda_2 \to 0$, regardless of the value of $z_2$. In this specific case, the existence of the derivative at the kink is due to the symmetry of the PDE and initial condition.
	
	$\,$
	
	\noindent
{\bf	Case 3:} $\mu_1 = 2$, $\mu_2 = 0$ (i.e., gradient of payoff at $45\deg$ angle to expansion direction).

	$\,$

\noindent
Here, again by elementary but lengthy calculation,
    	
	
    \begin{eqnarray*}
		\frac{\partial}{\partial \lambda_2} u(z,t,\lambda) 
		&=& \frac{\log{2}-z_1-z_2}{\sqrt{8\pi}} \frac{\sqrt{\lambda_1}}{\sqrt{\lambda_1+\lambda_2}^3} \exp\left(-\frac{(\log{2}-z_1-z_2)^2\lambda_2}{4t\lambda_1(\lambda_1+\lambda_2)}\right).
	\end{eqnarray*}
    In particular, the derivative vanishes at the kink where $z_1 + z_2 = \log{2}$. 
    Everywhere else, the derivative
    approaches a non-zero bounded value for  $\lambda_2\rightarrow 0$ for $\lambda_1>0$ fixed.
    \end{ex}

$\,$

    \begin{ex}[Standard call on geometric basket] \label{ex:GeometricBasket}
    
	Consider the same setting as in the previous Example \ref{ex:DigitalGeometricBasket}, but with payoff function
	\begin{eqnarray*}
	    h(s_1,s_2) &=& \max(s_1^{\mu_1}s_2^{\mu_2}-2 , 0), \text{ i.e.,}\\
	    g(z_1,z_2) &=& \max\left(\exp\left[\frac{\mu_1+\mu_2}{2}z_1 + \frac{\mu_1-\mu_2}{2}z_2\right] - 2 , 0 \right).
	\end{eqnarray*}

	$\,$
	
	\noindent
{\bf	Case 1:} For the first case, $\mu_1 = 1$, $\mu_2 = 1$, we have again
	\begin{eqnarray*}
		\frac{\partial}{\partial \lambda_2} u(z,t,\lambda) 
	    &=& 0.
	 \end{eqnarray*}

	$\,$
	
	\noindent
{\bf	Case 2:}
    For the second case, $\mu_1 = 1$, $\mu_2 = -1$, we find by straightforward calculus
    \begin{eqnarray}
	    \frac{\partial}{\partial \lambda_2} u(z,t,\lambda) 
	&=& \frac{t e^{z_2+\lambda_2 t}}{2} \left\{ - \frac{e^{-(z_2-\log{K}+2\lambda_2 t)^2/4t\lambda_2}}{\sqrt{\pi t\lambda_2}} + \Phi
	\left(\frac{z_2-\log{K}+2\lambda_2 t}{\sqrt{2t\lambda_2}}\right) \right\}, \label{eq:1stDiffFloor}
	\end{eqnarray}
	where $\Phi$ is the standard normal cdf.
	This derivative exists for all $z_2\in\mathbb{R}$ and all $\lambda_2 > 0$. For $\lambda_2 \rightarrow 0$ the limit exists and is well-defined provided $z_2\neq \log{2}$.
	
	If $z_2=\log{2}$, the right-hand side of $(\ref{eq:1stDiffFloor})$ does not converge for $\lambda_2 \rightarrow 0$. 
	We can, however, Taylor expand in $\sqrt{\lambda_2}$ instead of $\lambda_2$, since
	\[
	\frac{\partial}{\partial \sqrt{\lambda_2}} u(z_2,t,\lambda_2) = 2\sqrt{\lambda_2} \frac{\partial}{\partial \lambda_2} u(z_2,t,\lambda_2)
	\]
	 exists in the limit $\lambda_2\rightarrow 0$. This expansion will lead to an expansion error of size $O(\lambda_2)$ instead of $O(\lambda_2^2)$. Note that the PDE approximation method itself does not have to change.

	$\,$
	
	\noindent
{\bf	Case 3:} The third case is analytically lengthier than Example \ref{ex:DigitalGeometricBasket}, and we have therefore omitted it here. We refer to Section \ref{subsec:nonsmooth}, and particularly Examples \ref{ex:smoothing} and \ref{ex:smoothing2} for a qualitative discussion. The conclusion will be the same as in Case 3 of Example \ref{ex:DigitalGeometricBasket}, i.e., bounded derivative with respect to $\lambda_2$ for $\lambda_2\rightarrow 0$ for $\lambda_1>0$ fixed.

	\end{ex}


\end{document}